\documentclass{article}
\usepackage[utf8]{inputenc}
\usepackage{amsmath, amsthm, amssymb, amscd, amsxtra}
\usepackage[mathscr]{eucal}
\usepackage[all]{xy}
\usepackage{color}
\usepackage{authblk}
\usepackage{tikz}
\usetikzlibrary{decorations.markings}
\usepackage{tikz-cd}
\usepackage{arydshln}

\newcommand{\CS}{\text{CS}}
\newcommand{\Tor}{\text{Tor}}
\newcommand{\Tr}{\text{Tr}}
\newcommand{\Ima}{\text{Im}}
\newcommand{\Kernel}{\text{Ker}}

\newcommand{\Sym}{\text{Sym}}
\newcommand{\C}{\mathcal{C}}
\newcommand{\D}{\mathcal{D}}
\newcommand{\R}{\mathcal{R}}
\newcommand{\SU}{\text{SU}}
\newcommand{\bbZ}{\mathbb{Z}}
\newcommand{\bbR}{\mathbb{R}}
\newcommand{\bbC}{\mathbb{C}}
\newcommand{\IntSet}[1]{[0 \cdots #1]}
\newcommand{\floor}[1]{\lfloor #1 \rfloor}
\newcommand{\highlight}[1]{{\color{red} #1}}
\newcommand{\TLJ}{\text{TLJ}}
\newcommand{\unitobj}{\mathbf{1}}

\newcommand{\Hom}{\textrm{Hom}}

\newcommand{\SL}{\textrm{SL}}
\newcommand{\GL}{\textrm{GL}}
\newcommand{\bbS}{\mathbb{S}}
\newcommand{\so}{\mathfrak{so}}
\newcommand{\mcB}{\mathcal{B}}
\newcommand{\Adj}{\textrm{Adj}}
\newcommand{\nabchar}{\chi^{\text{nab}}}

\newcommand{\gtensor}{\boxtimes_{gr}}
\tikzset{
    myarrow/.style n args = {2}{
    postaction = decorate,
    decoration={
    markings,
    mark=at position {#1} with {\arrow{#2}}}
    }
}

\theoremstyle{definition}
\newtheorem{theorem}{Theorem}[section]
\newtheorem{proposition}[theorem]{Proposition}

\newtheorem{lemma}[theorem]{Lemma}
\newtheorem{lemma-definition}[theorem]{Lemma-Definition}

\newtheorem{definition}[theorem]{Definition}
\newtheorem{remark}[theorem]{Remark}

\newtheorem{example}[theorem]{Example}

\title{From Three Dimensional Manifolds to Modular Tensor Categories}
\author[1]{Shawn X. Cui}
\author[2]{Yang Qiu}
\author[3]{Zhenghan Wang}
\affil[1]{{\small Department of Mathematics and Department of Physics and Astronomy, Purdue University, West Lafayette, IN 47906}}
\affil[2]{{\small Department of Mathematics, University of California,  Santa Barbara, CA 93106}}
\affil[3]{{\small Microsoft Station Q and Dept. of Math., University of California,  Santa Barbara, CA 93106}}
\affil[ ]{{\small \it cui177@purdue.edu, yangqiu@math.ucsb.edu, zhenghwa@microsoft.com}}

\date{}

\begin{document}

\maketitle

\begin{abstract}
Using M-theory in physics,  Cho, Gang, and Kim (JHEP \textbf{2020}, 115 (2020) ) recently outlined a program that connects two parallel subjects of three dimensional manifolds, namely, geometric topology and quantum topology. They suggest that classical topological invariants such as Chern-Simons invariants of $\SL(2,\bbC)$-flat connections and $\SL(2,\bbC)$-adjoint Reidemeister torsions of a three manifold can be packaged together to produce a $(2+1)$-topological quantum field theory, which is essentially equivalent to a modular tensor category.  It is further conjectured that every modular tensor category can be obtained from a three manifold and a semi-simple Lie group.  In this paper, we study this program mathematically, and provide strong support for the feasibility of such a program.  The program produces an algorithm to generate the potential modular $T$-matrix and the quantum dimensions of a candidate modular data.
The modular $S$-matrix follows from essentially a trial-and-error procedure. We find premodular tensor categories that realize candidate modular data constructed from Seifert fibered spaces and torus bundles over the circle that reveal many subtleties in the program. We make a number of improvements to the program based on our examples.  Our main result is a mathematical construction of the modular data of a premodular category from each Seifert fibered space with three singular fibers and a family of torus bundles over the circle with Thurston SOL geometry.  The modular data of premodular categories from Seifert fibered spaces can be realized using Temperley-Lieb-Jones categories and the ones from torus bundles over the circle are related to metaplectic categories.   We conjecture that a resulting premodular category is modular if and only if the three manifold is a $\bbZ_2$-homology sphere, and condensation of bosons in the resulting properly premodular categories leads to either modular or super-modular tensor categories.
\end{abstract}

\tableofcontents

\section{Introduction}

There are two parallel  universes in three dimensional topology for the last several decades that do not intersect much: the classical Thurston world and the quantum Jones world.  One famous conjecture that hints a deep connection of the two worlds is the volume conjecture.  Recently M-theory in physics suggests another surprising different connection: classical topological invariants such as Chern-Simons invariants of $\SL(2,\bbC)$-flat connections and $\SL(2,\bbC)$-adjoint Reidemeister torsions of a three manifold $X$ can be packaged together to produce a $(2+1)$-topological quantum field theory (TQFT) \cite{gang20}, which is essentially equivalent to a modular tensor category \cite{turaevbook}.  It is further conjectured in \cite{gang20} that every modular tensor category can be obtained from a three manifold and a semi-simple Lie group.  In this paper, we study this program mathematically, and provide strong support for such a program.  The program as outlined in \cite{gang20} produces an algorithm to generate the potential modular $T$-matrix and the quantum dimensions of a candidate modular data.  The modular $S$-matrix follows from essentially a trial-and-error procedure. We find premodular tensor categories that realize candidate modular data from Seifert fibered spaces and torus bundles over the circle that reveal many subtleties in the program. Our main result is a mathematical construction of the modular data of a premodular category from each Seifert fibered space with three singular fibers and some torus bundles over the circle with Thurston SOL geometry.  The modular data of the premodular categories from Seifert fibered spaces can be realized using Temperley-Lieb-Jones categories and the ones from torus bundles over the circle are related to metaplectic categories \cite{wangbook,metaplectic}.   A more general study of the torus bundle case is in \cite{cui2021torus}.  We conjecture that the resulting premodular category is modular if and only if the three manifold is a $\bbZ_2$-homology sphere, and condensation of bosons in the resulting properly premodular categories leads to either modular or super-modular tensor categories.

The program from three manifolds to modular tensor categories is a far-reaching progeny of the mysterious six-dimensional super-symmetric conformal field theories (SCFTs) spawned by M-theory. 
Our strong support for the program indirectly provides evidence for these 6d SCFTs.
The dimension reduction or compactification of these 6d SCFTs to 3d depends on a three manifold $X$, and in general the resulting theory $T(X)$ is a super-conformal field theory.  When $X$ is non-hyperbolic, it is argued in \cite{gang20} that $T(X)$ flows to a TQFT in the infrared limit and super-symmetry is decoupled, thus we obtain a (2+1)-TQFT labeled by $X$, hence a MTC $\mathcal{B}_X$.  The program outlined in \cite{gang20} centers on an algorithm to produce the quantum dimensions and topological twists of a MTC, and a trial-and-error algorithm for the modular $S$-matrix. The assumption on the three manifolds $X$ in \cite{gang20} includes that $X$ is non-hyperbolic and the $\SL(2,\bbC)$ representation variety of the fundamental group $\pi_1(X)$ consists of finitely many conjugacy classes that all could be conjugated into either $SU(2)$ or $\SL(2,\bbR)$ subgroups of $\SL(2,\bbC)$.  Our examples show that all but the non-hyperbolic assumption can be dropped.  One subtlety is that we need to use indecomposable reducible representations in our torus bundle over the circle examples.  We do not know whether or not MTCs could be constructed from hyperbolic three manifolds as the program as we formulated in this paper is more flexible.  The main difficulty for more examples lies in the explicit calculation of Chern-Simons (CS) invariant and adjoint Reidemeister torsion of flat connections.

An $\SL(2,\bbC)$-representation of $\pi_1(X)$ is the same as a flat connection of the trivial $\SL(2,\bbC)$-bundle.  There are two well-known invariants for a flat connection: the Chern-Simons (CS) invariant and the adjoint Reidemeister torsion.  Each flat connection that satisfies certain conditions would give rise to an anyon type and the Reidemeister torsion is essentially the quantum dimension and the CS invariant is the conformal weight of the anyon.  

For each Seifert fibered space with three singular fibers, we define a potential modular data inspired by the many examples in \cite{gang20}.  All those modular data can be realized by premodular categories obtained as a $\bbZ_2$-graded product of Temperley-Lieb-Jones categories.  
We expect that our results can be easily generalized to any number of the singular fibers if the adjoint Reidemeister torsions of the $\SL(2,\bbC)$ flat connections can be calculated because the CS invariants in this case are known.  It is not clear if there are new MTCs among our examples.  Going beyond Seifert fibered spaces, we analyze some torus bundles over the circle and identify the resulting premodular categories as the integral subcategories of $SO(N)_2$ for odd $N$.
An important observation for the connection to Temperley-Lieb-Jones categories for Seifert fibered spaces is a relation between the slope of a singular fiber and the order of the Kauffman variable $A$ in Temperley-Lieb-Jones theories \cite{wangbook}.  Essentially the slope of a singular fiber determines a root of unity $A$, which allows us to realize all the candidate modular data from Seifert fibered spaces with three singular fibers.  

The content of the paper is as follows.  In Sec. 2, we outline our version of the program taking into account the many subtleties that we encountered in our examples.  We also recall the definition of CS invariant and adjoint Reidemeister torsion, and collect some known results of CS invariants of Seifert fibered spaces.   In Sec. 3, we study the Seifert fibered spaces and carry out the necessary calculations of CS and torsion invariants for our examples, and do the same for torus bundles over the circle in Sec. 4.  Finally, in Sec. 5, we discuss some future directions and open questions.

\section{A program from $3$-manifolds to modular categories}\label{sec:backgrounds}

The proposed program in \cite{gang20} from three manifolds to MTCs came from physics, and the paper provides an algorithm to produce the potential modular $T$-matrix and all quantum dimensions of a candidate modular data from irreducible representations of the fundamental groups of three manifolds to $\SL(2,\bbC)$.  Our results in Sec. 3 and Sec. 4 that realize candidate modular data from Seifert fibered spaces and torus bundles over the circle reveal many subtleties in the program as outlined in \cite{gang20}. In this section, we follow the overall program of \cite{gang20} and make a number of improvements to reformulate mathematically the construction of candidate modular data from three manifolds taking into account these new subtleties. 

\subsection{Representation and character variety}

Suppose $X$ is an orientable connected closed $3$-manifold and $G$ is a semi-simple Lie group.  The set of representations of the fundamental group $\pi_1(X)$\footnote{We omit the irrelevant base point.} to $G$ consists of all group homomorphisms from $\pi_1(X)$ to $G$, denoted by $\Hom(\pi_1(X),G)$,  up to conjugation. The representation variety $\R(X,G)$ of $\pi_1(X)$ to $G$ is simply $\Hom(\pi_1(X),G)//G$---equivalence classes of representations up to conjugation.

In this paper, we will mainly consider the case $G=\SL(2,\bbC)$ and its higher dimensional irreducible representations $\Sym^j$ of dimension $j+1$.
Given such a representation $\rho: \pi_1(X) \rightarrow \SL(n,\bbC)$, its  character is the function on $\pi_1(X)$ given by $\chi_\rho(x)=\Tr(\rho(x))$ for $x\in \pi_1(X)$.  The character variety $\chi(X, \SL(n,\bbC))$ of $X$ consists of all such character functions.  We will also denote the representation variety $\R(X, \SL(2,\bbC))$ and character variety $\chi(X, \SL(2,\bbC))$ simply as $\R(X)$ and $\chi(X)$.  In this paper, the topology of the spaces of the representation and character varieties is not important.

There are three obvious nontrivial automorphisms of $\SL(2,\bbC)$ by sending an element $g\in \SL(2,\bbC)$ to its complex conjuagte $g^{*}$, its transpose followed by inverse ${(g^{t}})^{-1}$, and the composition $(g^{\dag})^{-1}$ of the previous two operations.  For each representation of $\pi_1(X)$ to $\SL(2,\bbC)$, post-composing with one of the three automorphisms of $\SL(2,\bbC)$ gives rise to another representation, hence representations in $\R(X)$ come in group of four in general.
Another obvious way to change a representation $\rho$ in $\R(X,G)$ is to tensor $\rho$ with a representation of $\pi_1(X)$ to the center $Z(G)$ of $G$.  Representations of $\pi_1(X)$ to the center $Z(G)$ are in one-one correspondence with cohomology classes in the cohomology group $H^1(X,Z(G))$.

\subsection{Non-hyperbolic three manifolds}\label{subsec:seifert}

The proposed program in \cite{gang20} and in this section is to produce modular tensor categories (MTCs) from closed three manifolds and show that each MTC can be obtained from at least one three manifold.  It is known that different three manifolds can lead to the same MTC.  As suggested in \cite{gang20}, we will focus on non-hyperbolic three manifolds.  There are seven non-hyperbolic geometries and six can be realized by Seifert fibered three manifolds with the exception SOL \cite{scott83}.  The geometry $\bbS^2\times \bbR$ is not useful for our purpose as we need representations from the fundamental group to $\SL(2,\bbC)$ with non-Abelian images.  We will mainly consider Seifert fibered spaces in this paper, but in Sec. \ref{SOL}, we will also study torus bundles over the circle with SOL geometry and more subtleties arise.

Seifert fibered three manifolds $X$ are those that can be foliated into disjoint union of circles and are completely enumerated \cite{orlikbook}.  In this paper, all our three manifolds are orientable, and we will denote the Seifert fibered spaces (SFSs) by the notation $X=\{b;(o,g);(p_1,q_1),(p_2,q_2),\cdots,(p_n,q_n)\}$ as explained below.  The quotient space of a SFS $X$, called the base orbifold $B$, by sending each circle, called a fiber, to a point is a topological surface. The symbol $(o,g)$ means that the base topological surface $B$ is an orientable closed surface of genus $g$.  

Each fiber has a product neighbourhood $D^2\times S^1$ in the SFS $X$ except $n$ singular fibers labeled by $(p_i,q_i), i=1,\cdots, n$.
The neighborhood of the $i$-th singular fiber is obtained from $D^2\times [0,1]$ by identifying the point $(x,0), x\in D^2$ with the point $(r_{a_i,p_i}(x),1)$, where $r_{a_i,p_i}$ is the rotation of the disk $D^2$ by the angle $2\pi a_i/p_i$, where $a_i \in \bbZ$ satisfies $a_iq_i = 1 \mod p_i$. The pair of coprime integers $(p_i,q_i)$ are the corresponding surgery coefficient.   The fundamental group of $X$ fits into a short exact sequence $1\rightarrow \pi_1(F) \rightarrow \pi_1(X) \rightarrow \pi_1^{orb}(B)\rightarrow 1$, where $\pi_1(F)\cong \bbZ$ for a regular fiber $F\cong \bbS^1$ and $\pi_1^{orb}(B)$ is the orbifold fundamental group of $B$ (not the same as the fundamental group $\pi_1(B)$ of the topological surface $B$ in general).  The integer $b$ in the notation is the obstruction class, which is also the order of the generator of $\pi_1(F)$ in $\pi_1^{orb}(B)$.  Since we consider SFSs as three manifolds up to homeomorphism rather than as fibered spaces, we may always set $b$ to $0$.

The fundamental group of  $X=\{b;(o,g);(p_1,q_1),(p_2,q_2),\cdots,(p_n,q_n)\}$ has a presentation
\begin{multline}\label{equ:SFS_pi1_general}
    \pi_1(X)=\langle a_j,b_j,x_i,h,\  j= 1, \cdots, g, i = 1, \cdots, n\ |\\ [a_j,h]=[b_j,h]=[x_i,h]=x_i^{p_i}h^{q_i}=1,\ x_1\,\cdots\,x_n[a_1,b_1]\,\cdots\,[a_g,b_g]=h^b\rangle.
\end{multline}
In particular, the fundamental group of $X=\{0;(o,0);(p_1,q_1),(p_2,q_2),(p_3,q_3)\}$ with base $\bbS^2$ and three singular fibers, denoted simply as $\{b;(p_1,q_1),(p_2,q_2),(p_3,q_3)\}$ sometimes,  is 
$$\pi_1(X)=\langle x_1,x_2,x_3,h|x_i^{p_i}h^{q_i}=1,x_ih=hx_i,x_1x_2x_3=h^b\rangle.
$$
The orientable SFS $\{0;(o,0);(p_1,q_1),(p_2,q_2),\cdots,(p_n,q_n)\}$ with base $\bbS^2$ and $n$ singular fibers has a surgery diagram shown in Fig. \ref{fig:my_label}.

\begin{figure}
    \centering
\begin{tikzpicture}[scale=0.7]
\draw (0,0) arc (-170:170:1 and 2);
\draw (3,0) arc (-170:170:1 and 2);
\draw (6,0) arc (-170:170:1 and 2);
\draw (-1,0.5)--(1.5,0.5);
\draw (2.5,0.5)--(4.5,0.5);
\draw (5.5,0.5)--(7.5,0.5);
\draw (8.5,0.5)--(9.5,0.5);
\draw (-1,0.5)--(-1,-3);
\draw (-1,-3)--(9.5,-3);
\draw (9.5,-3)--(9.5,0.5);
\node[below] at (1,-1.6) {$\frac{p_1}{q_1}$};
\node[below] at (4,-1.6) {$\frac{p_2}{q_2}$};
\node[below] at (7,-1.6) {$\frac{p_n}{q_n}$};
\node at (4,-3.5) {0};
\draw (5.5,1) node{$\cdots$};
\end{tikzpicture}
    \caption{Surgery link of Seifert fibered space with base $\bbS^2$}
    \label{fig:my_label}
\end{figure}
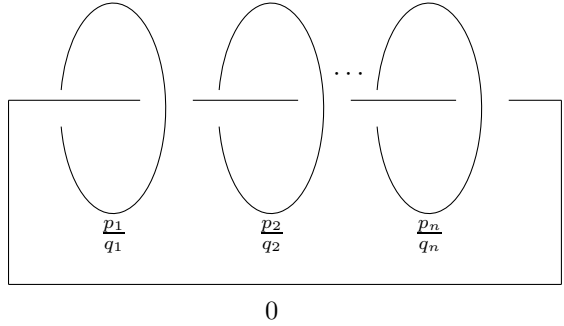

\subsection{Chern-Simons invariant}\label{subsec:CS_review}

Given an orientable connected closed three manifold $X$, a morphism $\rho$ of its fundamental group $\pi_1(X)$ to a semi-simple Lie group $G$ can be identified as the holonomy representation  of a flat connection $A_{\rho}$ on the trivial principal $G$-bundle over $X$.  Therefore, in the following we will use the terms a representation $\rho$ and a flat connection $A$  interchangeably via such an identification.  

Let $X$ be a closed 3-manifold and $\rho:\pi_1(X)\longrightarrow \SL(2,\mathbb{C})$ be a holonomy representation. Denote by $A_{\rho} $ the corresponding Lie algebra $\mathfrak{sl}(2,\mathbb{C})$-valued 1-form  on $X$. The Chern-Simons (CS) invariant of $\rho$ is defined as
\begin{equation}\label{equ:CS_integral_def}
    \CS(\rho)=\frac{1}{8\pi^2}\int_{X}\Tr(dA_{\rho}\wedge A_{\rho}+\frac{2}{3}A_{\rho}\wedge A_{\rho}\wedge A_{\rho}) \mod 1,
\end{equation}
where the integral with its coefficient in the front is   well-defined up to integers.

The CS  invariant $\CS(\rho)$ depends only on the character $\chi(\rho)$ of $\rho$ \cite{kirk93}, hence it descends from the representation variety $\R(X)$ to the character variety $\chi(X)$.

Auckly computed the CS invariant of SFSs for $\SU(2)$ representations in \cite{auckly94}. The CS invariant of SFSs for $\SL(2,\bbC)$ representations may be known to experts. However, to make the paper self-contained, we provide a proof to compute that using method from \cite{kirk93}.

\begin{proposition}\label{prop:SFS_general_CS}
Let $X=\{0;(o,g);(p_1,q_1),(p_2,q_2),\cdots,(p_k,q_k)\}$ be an SFS with the presentation of $\pi_1(X)$ given in Equation \ref{equ:SFS_pi1_general} with $b = 0$. 
Choose integers $s_j$ and $r_j$ such that $p_js_j - q_j r_j = 1$. Suppose $\rho:\pi_1(X)\rightarrow \SL(2,\mathbb{C})$ is non-Abelian such that $\Tr(\rho(x_j))=2\cos\frac{2\pi n_j}{p_j}$, then
\begin{align*}
\CS(\rho)=\left\{\begin{aligned}
\sum_{j=1}^k\ \frac{r_jn_j^2}{p_j} \mod 1,&&&\rho(h)=I
\\
\sum_{j=1}^k\ (\frac{r_jn_j^2}{p_j}-\frac{q_js_j}{4}) \mod 1,&&&\rho(h)=-I
\end{aligned}
\right.
\end{align*}
\end{proposition}
\begin{remark}
The formula for the CS invariant in Proposition \ref{prop:SFS_general_CS} differs from that in \cite{auckly94} with a negative sign. We believe this discrepancy is due to conventions.
\end{remark}
Before proving the proposition, we recall some facts in \cite{kirk93}

Let $T$ be a torus and consider $\chi(T)$, the character variety of $T$ to $\SL(2,\bbC)$. It is direct to see that $\chi(T)$ can be identified with $\Hom(\pi_1(T), \bbC^*)/\sim$, where $f \sim g$ if $f = g^{\pm 1}$ where $g^{-1}$ means point-wise inverse of $g$. We now describe a `coordinate-version' of $\chi(T)$.

Let $H$ be a group with the presentation,
$$H=\langle x,y,b\ |\ [x,y]=bxbx=byby=b^2=1\rangle,$$
and define an action of $H$ on $\mathbb{C}^2$ by
$$x(\alpha,\beta)=(\alpha+1,\beta),\ y(\alpha,\beta)=(\alpha,\beta+1),\ b(\alpha,\beta)=(-\alpha,-\beta).
$$
Denote the image of $(\alpha, \beta) \in \bbC^2$ in the quotient space $\bbC^2/H$ by $[\alpha,\beta]$.
Let $\Vec{v} = (v_1, v_2)$ be any $\bbZ$-basis of $H_1(T)$, and define the map,
\begin{equation*}
    f_{\Vec{v}}\colon \bbC^2/H \rightarrow \chi(T),
\end{equation*}
such that $f_{\Vec{v}}[\alpha, \beta] \in \chi(T)$ sends
\begin{equation*}
    v_1  \mapsto e^{2\pi i \alpha}, \quad v_2 \mapsto e^{2\pi i \beta}.
\end{equation*}
It can be checked that $f_{\Vec{v}}$ is a homeomorphism. A representation of $\pi_1(T)$ that induces the character $f_{\Vec{v}}[\alpha, \beta]$ is given by,
\begin{equation*}
    v_1 \mapsto 
    \begin{pmatrix}
    e^{2\pi i\alpha}&0\\0&e^{-2\pi i\alpha}
    \end{pmatrix}, \quad
    v_2 \mapsto
    \begin{pmatrix}
    e^{2\pi i\beta}&0\\0&e^{-2\pi i\beta}
    \end{pmatrix}.
\end{equation*}
Furthermore, the homeomorphism $f_{\Vec{v}}$ is natural in the following sense. Let $\Vec{w}$ be another basis such that $\Vec{w} = \Vec{v}A$ for some $A \in \GL(2,\bbZ)$ (viewing $\Vec{w}$ and $\Vec{v}$ as row vectors), and define the map $\Phi_{\Vec{v},\Vec{w}}\colon \bbC^2 \to \bbC^2$ by right multiplying (row) vectors of $\bbC^2$ by $A$ on the right. Then $\Phi_{\Vec{v},\Vec{w}}$ induces a homeomorphism, still denoted by $\Phi_{\Vec{v},\Vec{w}}$, from $\bbC^2/H$ to $\bbC^2/H$, and the following diagram commutes,
\begin{equation*}
    \begin{tikzcd}
    \bbC^2/H\arrow[r, "\Phi_{\Vec{v},\Vec{w}}"]\arrow[d, "f_{\Vec{v}}"] & \bbC^2/H \arrow[dl,"f_{\Vec{w}}"] \\
    \chi(T)&
    \end{tikzcd}
\end{equation*}
Hence, we think of each $\bbC^2/H$ with a choice of basis $\Vec{v}$ as a coordinate realization of $\chi(T)$. In fact, $\chi(T)$ is isomorphic to the direct limit\footnote{Here all maps involved are isomorphisms, so the notion of direct limit and inverse limit do not make a difference.} of $\{(\bbC^2/H)_{\Vec{v}}, \ \Phi_{\Vec{v},\Vec{w}}\}$,
\begin{equation*}
    \chi(T) \simeq \lim_{\longrightarrow}\, (\bbC^2/H)_{\Vec{v}},
\end{equation*}
where $(\bbC^2/H)_{\Vec{v}}$ is a copy of $\bbC^2/H$ indexed by $\Vec{v}$.

Next, we introduce a $\bbC^*$ bundle over $\chi(T)$. 
Define an action of $H$ on $\mathbb{C}^2\times\mathbb{C}^*$ lifting that on $\mathbb{C}^2$ by
\begin{align*}
   x(\alpha,\beta;z)&=(\alpha+1,\beta;ze^{2\pi i\beta}), \\ y(\alpha,\beta;z)&=(\alpha,\beta+1;ze^{-2\pi i\alpha}),\\ b(\alpha,\beta;z)&=(-\alpha,-\beta;z). 
\end{align*}
The canonical projection $\mathbb{C}^2\times\mathbb{C}^* \to \mathbb{C}^2$ induces a projection
\begin{align*}
    p \colon \mathbb{C}^2\times\mathbb{C}^*/H \to \mathbb{C}^2/H,
\end{align*}
which makes $\mathbb{C}^2\times\mathbb{C}^*/H$ a $\bbC^*$ bundle over $\mathbb{C}^2/H$. Given two bases $\Vec{v}, \ \Vec{w}$ of $H_1(T)$ with $\Vec{w} = \Vec{v}A$, $\Phi_{\Vec{v}, \Vec{w}}$ can be covered by a bundle isomorphism. Explicitly, define $\tilde{\Phi}_{\Vec{v}, \Vec{w}}\colon \mathbb{C}^2\times\mathbb{C}^*/H \to \mathbb{C}^2\times\mathbb{C}^*/H$ which maps $[\alpha,\beta;z]$ to $[(\alpha,\beta)A; z^{\det(A)}]$. Then the following diagram commutes,
\begin{equation}\label{equ:commu_diag}
    \begin{tikzcd}
    (\mathbb{C}^2\times\mathbb{C}^*/H)_{\Vec{v}} \arrow[d, "p"]\arrow[r, "\tilde{\Phi}_{\Vec{v}, \Vec{w}}"] & (\mathbb{C}^2\times\mathbb{C}^*/H)_{\Vec{w}}\arrow[d, "p"] \\
    (\mathbb{C}^2/H)_{\Vec{v}}\arrow[r, "\Phi_{\Vec{v}, \Vec{w}}"] & (\mathbb{C}^2/H)_{\Vec{w}}
    \end{tikzcd}
\end{equation}
Let $\tilde{E}(T)$ be the direct limit of $\{(\mathbb{C}^2\times\mathbb{C}^*/H)_{\Vec{v}},\ \tilde{\Phi}_{\Vec{v}, \Vec{w}} \}$.
Then Equation \ref{equ:commu_diag} induces a map $ p\colon \tilde{E}(T) \to \chi(T)$ which makes $\tilde{E}(T)$ a $\bbC^*$ bundle over $\chi(T)$, and the diagram below commutes,
\begin{equation*}
\begin{tikzcd}
    \tilde{E}(T)\arrow[d, "p"] & (\mathbb{C}^2\times\mathbb{C}^*/H)_{\Vec{v}} \arrow[d, "p"]\arrow[l,""]\\
    \chi(T) & (\mathbb{C}^2/H)_{\Vec{v}}\arrow[l,"f_{\Vec{v}}"]
\end{tikzcd}
\end{equation*}
We often represent an element of $\tilde{E}(T)$ by a `coordinate' $[\alpha,\beta;z]_{\Vec{v}}$ with respect to a basis $\Vec{v}$. Changing the basis to $\Vec{w} = \Vec{v}A$ induces the equality
\begin{equation*}
    [\alpha,\beta;z]_{\Vec{v}} = [(\alpha,\beta)A;z^{\det(A)}]_{\Vec{w}},
\end{equation*}
and when the bases involved are clear from the context, we will omit them.

We also need an `orientation-version' of $\tilde{E}(T)$. Now assume $T$ is {\it oriented}, and define $E(T)$ to be the direct limit of $\{(\mathbb{C}^2\times\mathbb{C}^*/H)_{\Vec{v}},\ \tilde{\Phi}_{\Vec{v}, \Vec{w}} \}$ where the limit is taken only over positive bases $\Vec{v}$ of $H_1(T)$, namely, those $\Vec{v}$ such that $v_1 \wedge v_2$ matches the orientation of $T$. Apparently, $E(T)$ and $E(-T)$ are both bundles over $\chi(T)$, and are both isomorphic to $\tilde{E}(T)$. However, it will be of conceptual convenience for latter calculations to distinguish $E(T)$ from $E(-T)$. 

There is a fiber-wise pairing $\langle\ ,\ \rangle$ defined on $ E(T) \times E(-T)$ as follows. Given $e \in E(T), \ e' \in E(-T)$ such that $p(e) = p(e')$, choose an arbitrary positive basis $\Vec{v} = (v_1,v_2)$ of $H_1(T)$ and hence $\Vec{v}':= (-v_1, v_2)$ is a positive basis of $H_1(-T)$, and write $e = [\alpha,\beta;z]_{\Vec{v}}$, $e' = [-\alpha, \beta;z']_{\Vec{v}'}$ (or $e' = [\alpha, -\beta;z']_{-\Vec{v}'}$). Then $\langle e,e'\rangle:= zz'$. It can be checked that the pairing is well defined.

Lastly, the above notions can be generalized to multiple tori in a natural way. Let $S = \sqcup_{i=1}^k T_i$ be a disjoint union of $k$ oriented tori. Then $\chi(S) = \chi(T_1) \times \cdots \times \chi(T_k)$. The group $H^{k}$ acts on $(\bbC^2)^{ k}$ component-wise and the quotient is a `coordinate-version' of $\chi(S)$. The action of $H^{k}$ can also be lifted to $(\bbC^2)^{ k} \times \bbC^*$ where the $i$-th component $H_i$ in $H^{k}$ acts on the $i$-th copy of $\bbC^2$ in $(\bbC^2)^{ k}$ times $\bbC^*$, and $E(T)$ is the quotient of $(\bbC^2)^{ k} \times \bbC^*$ by this action. For $n \leq k$, similar to the pairing above, there is a generalized `pairing':
\begin{equation*}
    E(T_1 \sqcup \cdots \sqcup T_k) \times E(-T_1 \sqcup \cdots \sqcup -T_m) \to E(T_{m+1} \sqcup \cdots \sqcup T_k) .  
\end{equation*}

With the above notations, we recall several theorems in \cite{kirk93}. Let $X$ be an oriented compact 3-manifold with toral boundaries $\partial X=\sqcup_{i=1}^k T_i$ and $\rho\colon \pi_1(X) \to \SL(2,\bbC)$ be a holonomy representation. It is well-known that $\CS(\rho)$ in Equation \ref{equ:CS_integral_def} is not well defined since $X$ has boundary. Let 
\begin{equation*}
    c_X(\rho) = e^{2\pi i\, \CS(\rho)}.
\end{equation*}

\begin{theorem}[Theorem 3.2 of \cite{kirk93}]\label{thm:kirk1}
The Chern-Simons invariant defines a lifting $c_X:\chi(X)\longrightarrow E(\partial X)$ of the restriction map $r$ from the character variety of $X$ to the character variety of $\partial X$,
\begin{equation*}
    \begin{tikzcd}
          & E(\partial X)\arrow[d, "p"] \\
\chi(X)\arrow[ru, "c_X"]\arrow[r, "r"]   & \chi(\partial X)   
    \end{tikzcd}
\end{equation*}
Moreover, if $Y = X_1\cup X_2$ is a closed oriented $3$-manifold such that $X_1$ and $X_2$ are glued along toral boundaries $\partial X_1 = -\partial X_2$, then for $\chi\in\chi(Y)$, we have
$$e^{2\pi i\, \CS(\chi)}=\langle c_{X_1}(\chi_1),c_{X_2}(\chi_2)\rangle,
$$
where $\chi_i$ denotes the restriction of $\chi$ on $X_i$.
\end{theorem}

The following theorem is also due to  \cite{kirk93} which the authors proved for the case of $\SU(2)$ representations (Theorem 2.7), but an almost identical proof also works for $\SL(2,\bbC)$ representations. 
\begin{theorem}\label{thm:kirk2}
Let X be an oriented 3-manifold with toral boundaries $\partial X=\sqcup_{i=1}^kT_i$ and $\rho(t)\colon \pi_1(X)\rightarrow \SL(2,\mathbb{C})$ be a path of representations. Let $(\alpha_i(t),\beta_i(t)) $ be a lift of  $\chi \circ \rho(t)|_{T_i}$ to $\mathbb{C}^{2}$ with respect to some basis of $H_1(T_i)$. Suppose
$$c_X(\rho(t))=[\alpha_1(t),\beta_1(t),\cdots,\alpha_k(t),\beta_k(t);z(t)]
$$ 
Then
$$z(1)z(0)^{-1}=\exp\left(2\pi i\sum_{j=1}^k\int_0^1(\alpha_j\frac{d\beta_j}{dt}-\beta_j\frac{d\alpha_j}{dt})\right)
$$
In particular, if $\rho(1)$ is the trivial representation, then
$$c_{X}(\rho(0))=\biggl[\alpha_1(0),\beta_1(0),\cdots,\alpha_k(0),\beta_k(0);\exp\bigl(-2\pi i\sum_{j=1}^k\int_0^1(\alpha_j\frac{d\beta_j}{dt}-\beta_j\frac{d\alpha_j}{dt})\bigr)\biggr]
$$
\end{theorem}

The following two facts are proved for $\SU(2)$ representations in \cite{kirk93} (Theorems 4.1 and 4.2, respectively). Similar methods combined with Theorems \ref{thm:kirk1} and \ref{thm:kirk2} above show that they also hold for $\SL(2,\bbC)$ representations.

\vspace{0.2cm}
\noindent\textbf{Fact 1} Let $X$ be an oriented 3-manifold with toral boundaries $\partial X=\sqcup_{i=1}^nT_i$. Assume $H_1(X)$ is torsion free. Choose a positive basis $(\mu_i,\lambda_i)$ for $H_1(T_i)$. Let $\{x_j\ | \ j=1,\cdots,m\}$ be a basis of $H_1(X)$ and $\mu_i=\sum a_{ij}x_j,\lambda_i=\sum b_{ij}x_j$. Suppose that $\rho:\pi_1(X)\rightarrow \SL(2,\mathbb{C})$ is an Abelian representation and  $\Tr(\rho(x_j))=e^{2\pi i \gamma_j} + e^{-2\pi i \gamma_j}$ for some $\gamma_j\in\mathbb{C}$. Then
$$c_X(\rho)=\left[\ \sum a_{1j}\gamma_j,\ \sum b_{1j}\gamma_j,\ \cdots,\ \sum a_{nj}\gamma_j,\ \sum b_{nj}\gamma_j;\ 1\ \right]
$$

\noindent\textbf{Fact 2 } Let $F$ be a genus $g$ oriented surface with $k$ punctures. The fundamental group of $F$ has the presentation,
$$\pi_1(F)=\langle a_1,b_1,\cdots,a_g,b_g,x_1,\cdots,x_k\ |\ [a_1,b_1]\cdots [a_g,b_g]x_1 \cdots x_k=1\rangle,
$$
where $x_j$ corresponds to the oriented boundary (induced from $F$) of the $j$-th puncture. Let $Y = F \times S^1$ be endowed with the product orientation and let $\tilde{h} = * \times S^1$ be the central element of $\pi_1(Y)$ corresponding to the oriented  $S^1$ component. Then $\partial Y = \sqcup_{j=1}^k T_j$ with $T_j$ the torus corresponding to the $j$-th puncture and $(x_j, \tilde{h})$ is a positive basis for $H_1(T_j)$. Suppose $\rho:\pi_1(Y)\longrightarrow \SL(2,\mathbb{C})$ is a non-Abelian representation, which implies  $\Tr(\rho(\tilde{h}))=2\cos2\pi\beta$ for some $\beta \in \{0, \frac{1}{2}\}$. Suppose $\Tr(\rho(x_j))=e^{2\pi i \alpha_j} + e^{-2\pi i \alpha_j}$ for some $\alpha_j \in \bbC$. Then
$$c_Y(\rho)=\biggl[\alpha_1,\beta,\cdots,\alpha_n,\beta;\exp\bigl(-2\pi i\beta\sum_{j=1}^k\alpha_j\bigr)\biggr].
$$
Note that $c_Y(\rho)$ does not change under the replacement of some $\alpha_j$ by $-\alpha_j$.

The rest of the subsection is devoted to the proof of Proposition \ref{prop:SFS_general_CS}.
\begin{proof}
Let $Y = F \times S^1$ be as in \textbf{Fact 2} above with the chosen generators $x_j$ and $\tilde{h}$. Set $h = \tilde{h}^{-1}$. Then $X$ is obtained from $Y$ by gluing $k$ solid tori where the $j$-th solid torus $A_j$ is glued along $T_j$ by sending the meridian to $x_j^{p_j}h^{q_j}$. The generators $x_j$ and $h$ match those as presented in Equation \ref{equ:SFS_pi1_general}.
Choose a meridian-longitude pair $(\mu_j, \lambda_j)$ for $A_j$ such that $(\mu_j, \lambda_j)$ is a positive basis of $H_1(\partial A_j)$. The gluing of $A_j$ to $Y$  provides the transition of basis,
\begin{equation*}
    (\mu_j, \lambda_j) = (x_j,h)\,
    \begin{pmatrix}
    p_j & r_j \\
    q_j & s_j 
    \end{pmatrix}.
\end{equation*}
Since $\rho$ is non-Abelian, $\rho(h)$ is $\pm I$. By assumption,
\begin{alignat*}{3}
\Tr(\rho(x_j)) &= \exp(\frac{2\pi i n_j}{p_j} )+\exp(-\frac{2\pi i n_j}{p_j} ), \quad& \Tr(\rho(h)) &= 2\cos (2\pi m), \ m = 0, \frac{1}{2}.\\
\end{alignat*}
Therefore,
\begin{align*}
c_{Y}(\rho)&=\bigl[\frac{n_1}{p_1}, -m,\cdots,\frac{n_k}{p_k},-m;\exp(2\pi i\,m\sum_{j=1}^k\frac{n_j}{p_j})\bigr]_{(x_1,-h;\cdots;x_k,-h)}
\\
c_{A_j}(\rho)&=[0,\frac{r_jn_j}{p_j}+s_j m;1]_{(\mu_j,\lambda_j)} \\
&= [-q_j(\frac{r_jn_j}{p_j}+s_j m) , r_jn_j+s_jp_jm;1]_{(x_j,h)} \\
&=[\frac{n_j}{p_j}-s_j\alpha_j, m+ r_j\alpha_j;1], \ (\text{setting } \alpha_j = n_j+q_jm)\\
&=\biggl[\frac{n_j}{p_j}-s_j\alpha_j, m; \exp\bigl(2\pi i(r_j\alpha_j)(\frac{n_j}{p_j}-s_j\alpha_j) \bigr)\biggr]\\
&=\biggl[\frac{n_j}{p_j}, m; \exp\bigl(2\pi i(r_j\alpha_j)(\frac{n_j}{p_j}-s_j\alpha_j)+ 2\pi i (s_j\alpha_j)m \bigr)\biggr]\\
\end{align*}
Note that the relation $x_j^{p_j} h^{q_j} = 1$ implies that $\alpha_j$ must be an integer. Applying the pairing on $c_Y(\rho)$ and each $c_{A_j}(\rho)$ one by one, we obtain,
\begin{align*}
    \CS(\rho) &= \sum_{j = 1}^k (r_j \alpha_j \frac{n_j}{p_j} + s_j\alpha_jm + m \frac{n_j}{p_j}) \\
    &= \sum_{j = 1}^k \bigl(\frac{r_jn_j^2}{p_j} + s_jm(n_j+\alpha_j)\bigr)\\
    &= \sum_{j = 1}^k (\frac{r_jn_j^2}{p_j} -s_jq_jm^2).
\end{align*}

\end{proof}

\subsection{Adjoint Reidemeister torsion}\label{subsec:adjoint_torsion_review}

The  Reidemeister torsion ($R$-torsion)  $\tau(X)$  of  a cellulation $K_X$ of a manifold $X$ uses the action of the fundamental group $\pi_1(X)$ on the universal cover $\widetilde{K_X}$ to measure the complexity of the cellulation of $X$.  It is a topological invariant of $X$ from determinants of matrices obtained from the incidences of the cells of $\widetilde{K_X}$.  The $R$-torsion makes essential use of the bases in the chain complex of the universal cover, while the homology and homotopy groups do not see the geometric information encoded in the based chain complex.   For our purpose, we need the non-Abelian generalization of $R$-torsion twisted by a representation $\rho: \pi_1(X)\rightarrow G$ for some semi-simple Lie group $G$, in particular the adjoint Reidemeister torsion for the adjoint representation of $\SL(2,\bbC)$. We recall some basics here, for more details, please refer to \cite{milnor66} and \cite{turaev01}.

Let 
$$C_*=(0\longrightarrow C_n\stackrel{\partial_n}{\longrightarrow}C_{n-1}\stackrel{\partial_{n-1}}{\longrightarrow}\ \cdots\ \stackrel{\partial_1}{\longrightarrow}C_0\longrightarrow0)$$
be a
chain complex of finite dimensional vector spaces over the field $\mathbb{C}$. Choose a basis $c_i$ of $C_i$ and a basis $h_i$ of the $i$-th homology group $H_i(C_*)$. The torsion of $C_*$ with respect to
these choices of bases is defined as follows. For each $i$, let $b_i$ be a set of vectors in $C_i$ such that $\partial_i(b_i)$ is a basis of $\Ima(\partial_i)$ and let $\tilde{h}_i$ denote a lift of $h_i$ in $\Kernel(\partial_i)$. Then the
set of vectors $\tilde{b}_i := \partial_{i+1}(b_{i+1}) \sqcup \tilde{h}_i \sqcup b_i$ is a basis of $C_i$. Let $D_i$ be the transition matrix from $c_i$ to $\tilde{b}_i$. To be specific, each column of $D_i$ corresponds to a vector in $\tilde{b}_i$ being expressed as a linear combination of vectors in $c_i$.  
Define the torsion
$$\tau(C_*,c_*,h_*):=\left|\prod_{i=0}^n \ \text{det}(D_i)^{(-1)^{i+1}}\right|
$$
\begin{remark}
A few remarks are in order.
\begin{itemize}
\item The torsion, as it is denoted, does not depend on the choice of $b_i$ and the lifting of $h_i$.
    \item Here we define the torsion as the norm of the usual torsion, thus we do not need to deal with sign ambiguities.
\end{itemize}
\end{remark}

Let $X$ be a finite CW-complex and $(V,\rho)$ be a homomorphism $\rho:\pi_1(X)\longrightarrow \SL(V)$. The
vector space $V$ turns into a left $\mathbb{Z}[\pi_1(X)]$-module. The universal cover $\tilde{X}$ has a natural CW structure from $X$, and its chain complex $C_*(\tilde{X})$ is a free \textit{left} $\mathbb{Z}[\pi_1(X)]$-module
via the action of $\pi_1(X)$ as covering transformations. View $C_*(\tilde{X})$ as a \textit{right} $\mathbb{Z}[\pi_1(X)]$-module by $\sigma.g := g^{-1}.\sigma$ for $\sigma \in C_*(\tilde{X})$ and $g \in \pi_1(X)$.
 We define the twisted chain complex $C_*(X;\rho):= C_*(\tilde{X})\otimes_{\mathbb{Z}[\pi_1(X)]}V$.
Let $\{e_{\alpha}^i\}_{\alpha}$ be the set of $i$-cells of $X$ ordered in an arbitrary way.
Choose a lifting $\tilde{e}_{\alpha}^i$ of $e_{\alpha}^i$ in $\tilde{X}$. It follows that $C_i(\tilde{X})$ is generated by $\{\tilde{e}_{\alpha}^i\}_{\alpha}$ as a free $\mathbb{Z}[\pi_1(X)]$-module (left or right). Choose a basis of $\{v_{\gamma}\}_{\gamma}$ of $V$. Then $c_i(\rho):= \{\tilde{e}_{\alpha}^i \otimes v_{\gamma}\}$ is a $\bbC$-basis of $C_i(X;\rho)$. 
\begin{definition}\label{def:adj_cyclic}
Let $\rho:\pi_1(X)\longrightarrow \SL(V)$ be a representation.
\begin{enumerate}
    \item  We call $\rho$ acyclic if $C_*(X;\rho)$ is acyclic. Assume $\rho$ is acyclic. The torsion of $X$ twisted by $\rho$ is defined to be,
    \begin{equation*}
        \tau(X;\rho):= \tau\biggl(C_*(X;\rho),\, c_*(\rho)\biggr).
    \end{equation*}
    \item Let $\Adj: \SL(V) \to \SL(\mathfrak{sl}(V))$ be the adjoint representation of $\SL(V)$ on its Lie algebra $\mathfrak{sl}(V)$. We call $\rho$ \textit{adjoint acyclic} if $\Adj \circ \rho$ is acyclic. Assume $\rho$ is adjoint acyclic.  Define the \textit{adjoint Reidemeister torsion} of $\rho$ to be,
    \begin{equation*}
        \Tor(X;\rho):= \tau(X; \Adj \circ \rho).
    \end{equation*}
\end{enumerate}
\end{definition}
\begin{remark}
In this paper, we will only deal with the adjoint Reidemeister torsion $\rho$. For that matter, we simply call it the torsion of $\rho$. When no confusion arises, we abbreviate  $\Tor(X;\rho)$ as $\Tor(\rho)$.
\end{remark}

The following tool will be useful in computing torsions.
\\
\textbf{Multiplicativity Lemma}
Let $0\longrightarrow C_*^{\prime}\longrightarrow C_*\longrightarrow C_*^{\prime\prime}\longrightarrow0$ be an exact sequence of chain complexes. Assume that $C_*,C_*^{\prime},C_*^{\prime\prime}$ are based by $c_*,c_*^{\prime},c_*^{\prime\prime}$, respectively, and
their homology groups based by $h_*,h_*^{\prime},h_*^{\prime\prime}$, respectively. Associated to the short exact sequence is the long exact sequence $H_*$ in homology
$$\cdots \longrightarrow H_j(C_*^{\prime})\longrightarrow H_j(C_*)\longrightarrow
H_j(C_*^{\prime\prime})\longrightarrow H_{j-1}(C_*^{\prime})\longrightarrow \cdots
$$
with the reference bases. For each $i$, identify $c'_i$ with its image in $C_i$ and arbitrarily choose a preimage $\tilde{c}^{\prime\prime}_i$ of $c^{\prime\prime}_i$ in $C_i$. If the transition matrix between the bases $c_i$ and $c'_i \sqcup \tilde{c}^{\prime\prime}_i$ has determinant $\pm 1$, we call 
$c_*,c_*^{\prime},c_*^{\prime\prime}$ compatible. In this case, we have
$$\tau(C_*,c_*,h_*)\ =\ \tau(C_*^{\prime},c_*^{\prime},h_*^{\prime})\ \tau(C_*^{\prime\prime},c_*^{\prime\prime},h_*^{\prime\prime})\ \tau(H_*,\{h_* \sqcup h_*^{\prime} \sqcup h_*^{\prime\prime}\}).
$$

\subsection{Modular data from three manifolds}

The modular data of an MTC or a pre-modular category consist of the modular $S$- and $T$- matrices. Given a three manifold $X$ with certain conditions,  \cite{gang20} contains an algorithm for choosing the $T$-matrix and  the first row of the $S$-matrix, i.e. all quantum dimensions.  The next step for the full $S$-matrix is a trial-and-error algorithm based on finding the right loop operators for each simple object.  When all the loop operators are given, then the modular data can be computed.  There are no general algorithms to define loop operators, but in the cases of SFSs and SOL manifolds, we find the relevant loop operators completely.

\subsubsection{From adjoint-acyclic non-Abelian characters to simple object types}

Each premodular category has a label set---the isomorphism classes of the simple objects, and a label is an isomorphism class of simple objects, so we will refer to a label also as a simple object type.  In physics, an anyon model is a unitary MTC and a label is called an anyon type or a topological charge.  

A candidate label from a three manifold $X$ and $\SL(2,\bbC)$ is morally an irreducible representation of the fundamental group $\pi_1(X)$ to $\SL(2,\bbC)$.  But the precise definition is more subtle and based on our examples later, we make the following definition.  In particular, we discover that reducible but indecomposable representations cannot be discarded and  play important roles in the construction of premodular categories from torus bundles over the circle.  Our definition is specific for representations to $\SL(2,\bbC)$ and we expect an appropriate generalization is needed for other Lie groups such as $\SL(n,\bbC), n\geq 3$.

\begin{definition}
Let $\chi \in \chi(X)$ be an $\SL(2,\bbC)$-character of a three manifold $X$.
\begin{itemize}
    \item $\chi$ is non-Abelian if at least one representation $\rho: \pi_1(X)\rightarrow \SL(2,\bbC)$  with character $\chi$ is non-Abelian, i.e. $\rho$ has non-Abelian image in $\SL(2,\bbC)$. The set of all non-Abelian characters of $X$ is denoted by $\nabchar(X)$.
    \item A non-Abelian character $\chi$ is {\it adjoint-acyclic} if  each non-Abelian representation $\rho: \pi_1(X)\rightarrow \SL(2,\bbC)$ with character $\chi$ is adjoint-acyclic, namely, the chain complex associated with the universal cover $\tilde{X}$ twisted by $\Adj \circ \rho$ is acyclic (see Definition \ref{def:adj_cyclic}), and furthermore, the adjoint Reidemeister torsion of all such non-Abelian representations with character $\chi$ are the same.
    \item A candidate label is an {\it adjoint-acyclic non-Abelian character.}
    \item A candidate label set $L(X)$ from a three manifolds $X$ is a finite set of adjoint-acyclic non-Abelian characters in $\chi(X)$ with a pre-chosen character such that the difference of the CS invariant of each character $L(X)$ with that of the pre-chosen character is a rational number. 

The pre-chosen character is the candidate tensor unit.
\end{itemize}
\end{definition}

Note that by definition, the  adjoint Reidemeister torsion is well-defined for adjoint-acyclic non-Abelian characters. The $\CS$ invariant only depends on characters, and is hence also well-defined for such characters.

In this paper, our candidate label set is in general maximal in the sense it consists of all the adjoint-acyclic non-Abelian characters of the given three manifold.  It is also true that  the CS invariants of all the candidate labels including the candidate tensor unit are all rational in our examples.  We are not aware of any example of a candidate label set for which not all CS invariants are rational numbers.

\subsubsection{Vacuum choices, loop operators, and modular data}

Each simple object $x$ of a premodular category $\mcB$ has a quantum dimension $d_x$ and a  topological twist $\theta_x$.  The set $Td(B):=\cup_{i\in L(\mcB)}\{d_{x_i}, \theta_{x_i}\}$ will be called the twist-dimension set of $\mcB$, where $L(\mcB)$ is the label set of $\mcB$ and $\{x_i, i\in L(\mcB)\}$ form a complete representative set of simple objects of $\mcB$.  A candidate label set of a three manifold $X$ will lead to a candidate twist-dimension set in the following.  

The choice of a tensor unit or vacuum from a collection of adjoint-acyclic non-Abelian characters is not unique in general and it is known that different choices could produce different premodular categories.  Once a vacuum is chosen, then the adjoint Reidemeister torsion of each character is scaled to the absolute value of normalized quantum dimension and the difference of the CS invariant of the character with that of the vacuum is the conformal weight of the simple object up to a sign\footnote{The sign and hence the negative sign in front of $\CS$ invariant below is not important and the choice is made to be the same as in \cite{gang20}.}.

Given a three manifold $X$ and a Lie group $G$, a central representation of $\pi_1(X)$ is a homomorphism from $\pi_1(X)$ to the center $Z(G)$ of $G$.
For $G=\SL(2,\bbC)$, a central representation of $\pi_1(X)$ is simply a homomorphism from $\pi_1(X)$ to $\bbZ_2$.  The group of central representations can be identified with $H^1(X,\bbZ_2)$.  A central representation $\sigma \in H^1(X,\bbZ_2)$ of $\pi_1(X)$ naturally acts on $\R(X)$ by tensoring $\rho \in \R(X)$, i.e. by sending $\rho$ to $\rho\otimes \sigma$. Moreover, this action induces an action of central representations on the character variety $\chi(X)$.

\begin{definition}

\begin{enumerate}
\item Given a candidate label set $L(X)$ from a three manifold $X$, a central representation $\sigma$ is {\it bosonic} with respect to $L(X)$ if the action of $\sigma$ keeps $L(X)$ invariant and preserves the CS invariant of every candidate label.  If  the action of $\sigma$ changes the CS invariants of all candidate labels in $L(X)$ by either $0$ or $\frac{1}{2}$, then $\chi$ is called {\it fermionic} if it is not bosonic.

\item Two candidate labels are centrally related if they are in the same orbit under the action of $H^1(X,\bbZ_2)$ and they have the same CS and torsion invariant.
\end{enumerate}

\end{definition}

Given a candidate label set $L(X)$ of $X$ that is invariant under the action of $H^1(X,\bbZ_2)$, the candidate symmetric center $s(X)$ consists of all characters in $L(X)$ that are centrally related to the candidate tensor unit. Let $G_0(X)$ be the maximal subgroup of $H^1(X,\bbZ_2)$ such that $G_0(X)$ maps the candidate tensor unit onto $s(X)$.
The action of $G_0(X)$ separates $L(X)$ into orbits $\{O_0,\cdots, O_m\}$, where each subset $O_i$ of $L(X)$ consists of candidate labels that are centrally related to each other, and $O_0$ is the subset for the candidate vacuum.

We often represent a candidate label (a character) by arbitrarily choosing a representative (a representation of $\pi_1(M)$) for it.

\begin{definition}\label{def:admissible_label_set}
A candidate label set $L(X)=\{\rho_\alpha\}$ of a three manifold $X$ with $\rho_0$ the candidate vacuum is {\it admissible} if the following two equations hold with the notations as above:
\begin{equation}
    \sum_{\rho_\alpha \in L(X)} \frac{1}{2\Tor(\rho_\alpha)}=1,
\end{equation}
\begin{equation}\label{equ:S00_general}
      \left|\sum_{\alpha} \frac{\exp(-2\pi i\, \CS(\rho_{\alpha}))}{2 \Tor(\rho_{\alpha})}\right|=
\frac{1}{s_L} \frac{\sqrt{|s(X)|}}{\sqrt{2\Tor(\rho_0)}},
\end{equation}
where $s_L=1$ if all central representations in $G_o(X)$ are bosonic and $s_L={\sqrt{2}}$ if there is a fermionic one. 
\end{definition}

The conditions above are derived from the conjecture that the Mueger center of the potential premodular category is a collection of Abelian anyons parameterized by the subset $O_0$.  In the condensed category, each subset $O_i$ will be identified into a single composite object which has the same quantum dimension as that of any simple object in $O_i$ and which splits into a number of simple objects of the same quantum dimension. The resulting condensed category is either modular or super-modular depending on if there is a fermion in the candidate Mueger center. In a particular case when $X$ is a $\bbZ_2$ homology sphere, that is, $H^1(X,\bbZ_2) = 0$, Equation \ref{equ:S00_general} reduces to,
\begin{equation}\label{equ:S00}
      \left|\sum_{\alpha} \frac{\exp(-2\pi i\, \CS(\rho_{\alpha}))}{2 \Tor(\rho_{\alpha})}\right| =  \frac{1}{\sqrt{2 \Tor(\rho_0)}}.
\end{equation}

Given an admissible candidate label set $L(X)$ with the chosen candidate tensor unit $\rho_0$, then the candidate twist-dimension set is constructed as follows:
\begin{align}
         \theta_\alpha &=e^{-2\pi i (\CS(\rho_\alpha)-\CS(\rho_0))},\label{equ:CS_is_twist}\\
         D^2&=2\Tor(\rho_0)\label{equ:torsion0_is_D}\\
         d_\alpha^2 &=\frac{D^2}{2 \Tor(\rho_\alpha)},\label{equ:torsion_is_qdim}
\end{align}
where $D^2$ is the total dimension squared of the candidate premodular category.

Next, we discuss the construction of the $S$-matrix.
\begin{definition}
Given a three manifold $X$, a primitive loop operator of $X$ is a pair $(a,R)$, where $a$ is a conjugacy class of the fundamental group $\pi_1(X)$ of $X$ and $R$ a finite dimensional irreducible representation of $\SL(2,\bbC)$.
\end{definition}

Given an $\SL(2,\bbC)$-representation $\rho$ of $\pi_1(X)$ and a primitive loop operator $(a,R)$, then the weight of the loop operator $(a,R)$ with respect to $\rho$ is $W_{\rho}(a,R):=\Tr_R(\rho(a))$.  Denote by $\Sym^j$ the unique $(j+1)$-dimensional irreducible representation of $\SL(2,\bbC)$. Then $W_{\rho}(a,\Sym^j)$ can be computed from the Chebyshev polynomial $\Delta_j(t)$ defined recursively by,
\begin{equation}
    \Delta_{j+2}(t) = t\Delta_{j+1}(t) - \Delta_j(t), \quad \Delta_0(t)=1, \Delta_1(t) = t.
\end{equation}
Explicitly, 
\begin{equation}
    W_{\rho}(a,\Sym^j) = \Delta_j(t), \quad t =  W_{\rho}(a,\Sym^1) = \Tr(\rho(a)).
\end{equation}
From the above two equations, it follows that $W_{\rho}(a,\Sym^j)$ only depends on the character $\chi$ induced by $\rho$. It is direct to check that,
\begin{equation}\label{equ:Delta}
    \Delta_j(2\cos \theta) = \sin((j+1)\theta)/\sin \theta, \quad \Delta_j(-t) = (-1)^j \Delta_n(t).
\end{equation}

A fundamental assumption in constructing the $S$-matrix is that each candidate label $\rho_{\alpha}$ should correspond to a finite collection of primitive loop operators:
\begin{equation}\label{equ:local_operators}
    \rho_{\alpha} \mapsto \{(a^{\kappa}_{\alpha}, R^{\kappa}_{\alpha})\}_{\kappa}.
\end{equation}
By choosing a sign $\epsilon = \pm 1$, we define the $W$-symbols
\begin{equation}\label{equ:Wbetaalpha}
    W_{\beta}(\alpha)\  :=\  \prod_{\kappa} W_{\epsilon\, \rho_{\beta}}(a_{\alpha}^{\kappa},R_{\alpha}^{\kappa})\  = \ \prod_{\kappa} \Tr_{R^{\kappa}_{\alpha}}(\epsilon\, \rho_{\beta}(a^{\kappa}_{\alpha})), \quad \rho_{\alpha}, \rho_{\beta} \in L(X).
\end{equation}
The $W$-symbols and the un-normalized $S$-matrix $\tilde{S} = D\, S$ are related by,
\begin{equation}\label{equ:SW_relation}
    W_{\beta}(\alpha) \ = \  \frac{\tilde{S}_{\alpha\beta}}{\tilde{S}_{0\beta}} \quad \text{or} \quad \tilde{S}_{\alpha\beta}\  = \  W_{\beta}(\alpha)W_{0}(\beta),
\end{equation}
where $0$ denotes the tensor unit $\rho_0$. In particular, the quantum dimension 
\begin{equation}\label{equ:d_is_W0}
    d_{\alpha} = W_{0}(\alpha).
\end{equation}

\begin{remark}
Currently it involves a guess-and-trial process to find the correspondence between candidate labels and loop operators. One on hand, we know the absolute value of $W_0(\alpha)$ from Equations \ref{equ:torsion_is_qdim} and \ref{equ:d_is_W0}. On the other hand, $W_0(\alpha)$ can also be computed by Equation \ref{equ:Wbetaalpha}.  Note that, to obtain $W_0(\alpha)$, only the loop operators corresponding to $\rho_{\alpha}$ are required. Hence, we can choose a set of loop operators for  $\rho_{\alpha}$ so that the two ways of computing $W_0(\alpha)$ match (in absolute value). See the next remark for further validation of choices of loop operators. We leave it as a future direction to define the rigorous correspondence.
\end{remark}

\begin{remark}
 We expect that the resulting modular data corresponds to a MTC if and only if $H^1(X, \bbZ_2) = 0$. Note that, this is  a  purely topological condition, independent of the choice of loop operators. Hence, if $H^1(X, \bbZ_2) = 0$, we can also validate a choice of the loop operators by checking whether the resulting $S$ and $T$ matrices define a representation to $\SL(2,\bbZ)$. 
\end{remark}

 Before closing this section, we summarize the construction of modular data and show various choices during the procedure.  Given a closed 3-manifold $X$, first we choose a candidate label set $L(X)$ which is a finite set of adjoint-acyclic non-Abelian characters. A special character in $L(X)$ is prechosen as the tensor unit. The candidate label set is required to be admissible (see Definition \ref{def:admissible_label_set}). Then the topological twists ($T$-matrix) and quantum dimension squared of simple objects are given in Equations \ref{equ:CS_is_twist}-\ref{equ:torsion_is_qdim}. Next, we associate to each character in $L(X)$ a finite collection of primitive loop operators (see Equation \ref{equ:local_operators}) from which and a choice of $\epsilon = \pm 1$ the $W$-symbols are defined. The $S$-matrix and the $W$-symbols are related to each other by Equation \ref{equ:SW_relation}.

\section{Modular tensor categories from Seifert fibered spaces}
\label{sec:MTC_SFS}
In this section, we consider SFSs with three singular fibers and construct modular data associated with premodular categories.
Throughout the section, set $M = \{0;(o,0); (p_1, q_1), (p_2, q_2),  (p_3, q_3)\}$, where each pair $(p_k, q_k)$ are co-prime. So the underlying 2-manifold of the orbit surface $\Sigma$ has genus $0$ and both $M$ and $\Sigma$ are orientable.

\subsection{Character varieties of Seifert fibered spaces}
\label{subsec:variety}
For $M = \{0;(o,0); (p_1, q_1), (p_2, q_2),  (p_3, q_3)\}$, its fundamental group has the following presentation,
\begin{align*}
    \pi_1(M) \ =\ \langle\, x_1,x_2,x_3,h\, |\, x_k^{p_k}h^{q_k}=1,\, x_kh=hx_k,\, x_1x_2x_3=1, k = 1,2,3\,\rangle
\end{align*}
We look for all non-Abelian characters of $\pi_1(M)$ to $G = \SL(2, \bbC)$.

Let $\rho: \pi_1(M) \rightarrow G$ be a non-Abelian representation. Since $h$ is in the center of $\pi_1(M)$ and $\rho$ is non-Abelian, $\rho(h)$ must be $\pm I$. It follows that each $\rho(x_k)$ has finite order, and is diagonalizable in particular. Moreover, any $\rho(x_k)$ does not commute with another $\rho(x_j)$. This implies neither $\rho(x_k)$ can be $\pm I$. Up to conjugation, we assume $\rho(x_k)$ take the following form (writing $\rho(x_k)$ simply as $x_k$),
\begin{equation}\label{equ:xk matrix}
    x_1=
    \setlength\arraycolsep{2pt}
    \begin{pmatrix}
       e^{i\alpha_1}&0\\
       0&e^{-i\alpha_1}
    \end{pmatrix},
    x_2=
    \begin{pmatrix}
    a&b\\
    c&d
    \end{pmatrix}
    \sim
    \begin{pmatrix}
    e^{i\alpha_2}&0\\
    0&e^{-i\alpha_2}
    \end{pmatrix},
    x_3 \sim
    \begin{pmatrix}
     e^{i\alpha_3}&0\\
     0&e^{-i\alpha_3}
    \end{pmatrix}
\end{equation}
where $0 < \alpha_k < \pi$, $ad - bc =1$, and $b$ and $c$ are not simultaneously zero. We have the following linear equations for $a$ and $d$.
\begin{alignat}{3}
&\Tr(x_2) &&=e^{i\alpha_2}+e^{-i\alpha_2}&&=a+d \label{equ:ad equation1}\\
&\Tr(x_3) &&=e^{i\alpha_3}+e^{-i\alpha_3}&&=ae^{i\alpha_1}+de^{-i\alpha_1} \label{equ:ad equation2}
\end{alignat}
Hence, given the $\alpha_k'$s, or equivalently $\Tr(x_k)$, $a$ and $d$ are uniquely determined, and $a = \bar{d}$. Moreover, when $|a| \neq 1$ implying $bc\neq 0$, this also determines $\rho$ up to conjugacy. When $|a|=1$ implying $bc = 0$, there are precisely two conjugacy classes with
\begin{equation}
    x_2=
    \begin{pmatrix}
    a&1\\
    0&\bar{a}
    \end{pmatrix}
\text{ or }
x_2=
\begin{pmatrix}
a&0\\
1&\bar{a}
\end{pmatrix}
\label{reducible}
\end{equation}
It can be checked that these two representations are complex conjugate to each other up to conjugacy, and that their characters take real values. They give rise to the same character. There are two types of non-Abelian representations. One type is irreducible satisfying $b,c\neq 0$. Characters of representations of this type one-to-one correspond to conjugacy classes of representations \cite{culler83}. The other type is reducible with exactly one of $b,c$ zero. Each character of this type corresponds to two conjugacy classes.

To summarize, the triple $(\alpha_1,\alpha_2,\alpha_3)$ and $\Tr(h)$ uniquely determine the character. Next, we find all possible such triples.

If $h = I$, each $e^{i \alpha_k}$ is a $p_k$-th root of $1$. If $h = -I$, then $e^{i \alpha_k}$ is a $p_k$-th root of $1$ if $q_k$ is even, and a $p_k$-th root of $-1$ if $q_k$ is odd. We claim all triples satisfying the above conditions can be realized by some representations. Indeed, given such a triple $(\alpha_1,\alpha_2,\alpha_3)$, we define $\rho(x_1)$ and $\rho(x_2)$ as in Equation \ref{equ:xk matrix} and let $\rho(x_3):= (\rho(x_1)\rho(x_2))^{-1}$. Equations \ref{equ:ad equation1}, \ref{equ:ad equation2} determine $a$ and $d$, and we arbitrarily choose $b$ and $c$ such that $ad-bc = 1$. Again, Equations \ref{equ:ad equation1}, \ref{equ:ad equation2} guarantee that $\rho(x_k)$ so defined has eigenvalues $e^{\pm i \alpha_k}$, and therefore they satisfy all the relations in the presentation of $\pi_1(M)$.

Set $\alpha_k = \frac{2 \pi n_k}{p_k}$ and $\rho(h) = e^{2 \pi i \lambda}I$, $\lambda = 0, \frac{1}{2}$. If $\lambda = 0$ or if $\lambda=\frac{1}{2}$ and $q_k$ is even, then $n_k$ is an integer strictly between 0 and $\frac{p_k}{2}$. If $\lambda=\frac{1}{2}$ and $q_k$ is odd, then $n_k$ is a proper half integer strictly between 0 and  $\frac{p_k}{2}$. The quadruple $(n_1, n_2, n_3, \lambda)$ completely characterizes a character.

For an integer $p > 0$, denote by $\IntSet{p}$ the set of integers $\{0,1, \cdots, p\}$, and by $\IntSet{p}^e$ (resp. $\IntSet{p}^o$) the subset of even (resp. odd) integers in $\IntSet{p}$. The non-Abelian character variety of $M$ is given as follows,
\begin{equation}\label{equ:R0M_original}
\begin{split}
    \nabchar(M) = &\left\{\left(\frac{j_1+1}{2}, \frac{j_2+1}{2}, \frac{j_3+1}{2}, \frac{1}{2}\right)\ | \ j_k \in \IntSet{p_k-2}^{\epsilon_k}\right\}  \\
            \sqcup &\left\{\left(\frac{j_1+1}{2}, \frac{j_2+1}{2}, \frac{j_3+1}{2}, 0\right)\ | \ j_k \in \IntSet{p_k-2}^{o}\right\},
\end{split}
\end{equation}
where $\epsilon_k = ` e$' if $q_k$ is odd, and $\epsilon_k = \lq o $' otherwise. For $(n_1,n_2,n_3,\lambda) \in \nabchar(M)$, a corresponding representation $\rho$ has $e^{\pm \frac{2 \pi i n_k}{p_k}}$ as the eigenvalue of $\rho(x_k)$ and $\rho(h) = e^{2\pi i \lambda} I$.

The size of $\nabchar(M)$ is 
$$|\nabchar(M)| = \floor{\frac{p_1}{2}}\floor{\frac{p_2}{2}}\floor{\frac{p_3}{2}} + \floor{\frac{p_1-1}{2}}\floor{\frac{p_2-1}{2}}\floor{\frac{p_3-1}{2}},$$
where $\floor{x}$ is the greatest integer less than or equal to $x$. 

For instance, if all the $q_k'$s are odd, then $\nabchar(M)$ can also be written as,
\begin{align*}
      \nabchar = &\biggl\{\bigl(\frac{j_1+1}{2}, \frac{j_2+1}{2}, \frac{j_3+1}{2}, \frac{(j_1+1) \mod 2}{2}\bigr) \nonumber\\ 
         &\ | \ j_k \in \IntSet{p_k-2},\  j_1 = j_2 = j_3 \mod 2\biggr\}
\end{align*}

\subsection{Torsion of Seifert fibered spaces}
\label{subsec:SFS_torsion}

Freed computed torsions of Brieskorn homology spheres for the adjoint representations of irreducible $\SU(2)$ representations in \cite{freed92}. Kitano computed torsions of SFSs for irreducible $\SL(2,\mathbb{C})$ representations in \cite{kitano94}. However, we need to compute torsions of SFSs for the adjoint representations of nonAbelian $SL(2,\mathbb{C})$ representations containing both irreducible and reducible ones. This may be known to experts, but we did not find a reference for explicitly doing so.  To make the paper self-contained, we provide a detailed derivation of these torsions, generalizing the work of \cite{freed92} and \cite{kitano94}.

Let $X$ be the SFS  $\{0;(o,0); (p_1,q_1),(p_2,q_2),(p_3,q_3)\}$. Decompose $X$ as $\cup_{i=0}^3\,A_i\ \cup B$ along $\cup_{i=0}^3\,T_i$ where $B=(S^2-4pts)\times S^1$, and $A_0,\ A_i(i=1,2,3)$ are solid tori attached to $B$ by index $1,\ \frac{p_i}{q_i}$ along $T_0,\ T_i$, respectively. Let $\rho:\pi_1(X)\longrightarrow \SL(2,\mathbb{C})$ be a non-Abelian representation, $V=\mathfrak{sl}(2,\mathbb{C})$ be the adjoint representation of $\rho$ with the basis
$$e_1=\begin{pmatrix}0&1\\0&0\end{pmatrix},e_2=\begin{pmatrix}1&0\\0&-1\end{pmatrix},e_3=\begin{pmatrix}0&0\\1&0\end{pmatrix}
$$
From Section \ref{subsec:variety}, $\rho$ is parametrized by $(n_1,n_2,n_3,h)$ where $0<n_i<\frac{p_i}{2}$, $n_i\in\frac{1}{2}\mathbb{Z}$, $h=0,\frac{1}{2}$. Assume that $r_i,s_i\in\mathbb{Z}$, such that $p_is_i-r_iq_i=1$.
\begin{proposition}\label{prop:SFS_torsion}
When $\rho$ is nonAbelian, $C_*(\tilde{X})\otimes_{\mathbb{Z}[\pi_1(X)]}V$ is acyclic and
$$\Tor(X;\rho)=\frac{p_1p_2p_3}{\prod_{i=1}^34\sin^2\frac{2\pi r_in_i}{p_i}}
$$
\end{proposition}
\begin{proof}
Denote $C_*\otimes_{\mathbb{Z}[\pi_1(X)]}V$ by $C_{*,\rho}$, twisted homology by $H_*$, and the matrix of element in $\pi_1$ under $\rho$ by the same letter.

Given CW structure on $X$, we have the following exact chain sequence
$$0\longrightarrow\bigoplus_{i=0}^3C_{*,\rho}(T_i)\longrightarrow\bigoplus_{i=0}^3C_{*,\rho}(A_i)\oplus C_{*,\rho}(B)\longrightarrow C_{*,\rho}(X)\longrightarrow0
$$
and long exact sequence
\begin{multline*}
    0\longrightarrow\bigoplus_{i=0}^3 H_3(T_i)\longrightarrow\bigoplus_{i=0}^3 H_3(A_i)\oplus H_3(B)\longrightarrow H_3(X)\longrightarrow \cdots\\ 
    \longrightarrow \bigoplus_{i=0}^3H_0(T_i)\longrightarrow\bigoplus_{i=0}^3H_0(A_i)\oplus H_0(B)\longrightarrow H_0(X)\longrightarrow 0
\end{multline*}

Since the Reidemeister torsion is invariant under simple homotopy, we can just consider the simple homotopy types of the above spaces. $A_i$ is simple homotopy equivalent to a 1-complex, and $T_i$ and $B$ are each simple homotopy equivalent to a 2-complex. Thus we have 
$$H_3(A_i)=H_3(B)=H_3(T_i)=0,\ H_2(A_i)=0.
$$
Construct their cell structure as follows.
$$C_0(B)=<v_B>,C_0(T_i)=<v_{T_i}>,C_0(A_i)=<v_{A_i}>
$$
$$C_1(B)=<x_1,x_2,x_3,h>,C_1(T_i)=<m_i,l_i>,C_1(A_i)=<b_i>
$$
$$C_2(B)=<u_{1,B},u_{2,B},u_{3,B}>,C_2(T_i)=<u_{T_i}>
$$

where $v_*$ are base points of connected spaces, $x_i$ generate $\pi_1(S^2-4pts)$, $h=*\times S^1\in\pi_1(S^2-4pts\times S^1)$, $m_i,l_i$ are meridians and longitudes of $T_i$ respectively, $b_i$ are longitudes of boundary of $A_i$, $u_{i,B}$ are squares with boundary $x_ihx_i^{-1}h^{-1}$, $u_{T_i}$ are squares with boundary $m_il_im_i^{-1}l_i^{-1}$. $T_i(i=1,2,3)$ are attached to $x_i\times h$ by identity map and boundary of $A_i$ by $\begin{pmatrix}s_i&-q_i\\-r_i&p_i\end{pmatrix}$. $T_0$ is attached to $x_1x_2x_3\times h$ and boundary of $A_0$ by identity map. $x_1,x_2,x_3,h$ generate $\pi_1(X)$ as follows.
$$\pi_1(X)=<x_1,x_2,x_3,h|x^{p_i}h^{q_i}=1,x_ih=hx_i,x_1x_2x_3=1>
$$
For matrix under $\rho$, we have
$$x_i\sim\begin{pmatrix}\zeta_i&0&0\\0&1&0\\0&0&\zeta_i^{-1}\end{pmatrix},h=I
$$
where $\zeta_i$ is a $p_i$-th root of unity. $m_i=x_i$, $b_i=x_i^{r_i}$, $l_i=h$. Here we use 1-cell with ends points attached as element in $\pi_1$.

The work of \cite{freed92} can be generalized to irreducible representations of $SL(2,\mathbb{C})$. Thus we focus on reducible and nonAbelian representations. According to \ref{reducible}, taking upper triangular ones for example, they have the following form.
$$x_1=\begin{pmatrix}a_1&0\\0&a_1^{-1}\end{pmatrix},x_2=\begin{pmatrix}a_2&1\\0&a_2^{-1}\end{pmatrix},x_3=\begin{pmatrix}a_1^{-1}a_2^{-1}&-a_1\\0&a_1a_2\end{pmatrix}
$$
where $a_1,a_2,a_3=a_1^{-1}a_2^{-1}$ are roots of 1 or $-1$.
\\
For adjoint representation, we have
\begin{align}\label{adr}
x_1&=\begin{pmatrix}a_1^{-2}&0&0\\0&1&0\\0&0&a_1^2\end{pmatrix},x_2=\begin{pmatrix}a_2^{-2}&2a_2^{-1}&-1\\0&1&-a_2\\0&0&a_2^2\end{pmatrix}\notag
\\
x_3&=\begin{pmatrix}a_1^2a_2^2&-2a_2&-a_1^{-2}\\0&1&a_1^{-2}a_2^{-1}\\0&0&a_1^{-2}a_2^{-2}\end{pmatrix}
\end{align}

Let $w_i^{\pm}$ be the eigenvectors of $x_i$ for eigenvalue $\zeta_i=a_i^{-2}=e^{\frac{4\pi in_i}{p_i}},\zeta_i^{-1}=e^{-\frac{4\pi in_i}{p_i}}$ respectively and $w_i^0$ be the eigenvector of $x_i$ for eigenvalue 1. Then $w_i^{\pm}$ are the eigenvectors of $x_i^r$ for $\zeta_i^{r_i}$ and $w_i^0$ be the eigenvector of $x_i^{r_i}$ for 1. By scaling, assume that $|[w_i^{\pm}w_i^0]|=1$ in $V$. According to \ref{adr} , $w_1^{\pm},w_2^-$ is a basis of $V$. Similarly, for lower triangular ones in \ref{reducible}, $w_1^{\pm},w_2^+$ is a basis of $V$.

For $T_i(i=1,2,3)$, we have
$$0\longrightarrow C_{2,\rho}(T_i)\stackrel{\partial_2}{\longrightarrow}C_{1,\rho}(T_i)\stackrel{\partial_1}{\longrightarrow}C_{0,\rho}(T_i)\longrightarrow0
$$
where
$$\partial_2=\begin{pmatrix}O\\x_i-I\end{pmatrix},\partial_1=\begin{pmatrix}x_i-I&O\end{pmatrix}
$$
We have
\begin{align*}
H_2(T_i)&=<\tilde{u}_{T_i}\otimes w_i^0>
\\
H_1(T_i)&=<\tilde{m}_i\otimes w_i^0,\tilde{l}_i\otimes w_i^0>
\\
H_0(T_i)&=<\tilde{v}_{T_i}\otimes w_i^0>
\end{align*}

Choose preference basis $h_*$ for $H_*(T_i)$ as above and similarly with others. Without confusion, we omit $h_*$ in the expression as $c_*$.
\begin{align}
\tau(C_{*,\rho}(T_i))&=|\frac{[\tilde{l}_i\otimes(x_i-I)w_i^{\pm},\tilde{m}_i\otimes w_i^0,\tilde{l}_i\otimes w_i^0,\tilde{m}_i\otimes w_i^{\pm}]}{[\tilde{u}_{T_i}\otimes w_i^0,\tilde{u}_{T_i}\otimes w_i^{\pm}][\tilde{v}_{T_i}\otimes w_i^0,\tilde{v}_{T_i}\otimes(x_i-I)w_i^{\pm}]}|
\notag\\
&=|\frac{[\tilde{l}_i\otimes(\zeta_i^{\pm1}-1)w_i^{\pm},\tilde{m}_i\otimes w_i^0,\tilde{l}_i\otimes w_i^0,\tilde{m}_i\otimes w_i^{\pm}]}{[\tilde{u}_{T_i}\otimes w_i^0,\tilde{u}_{T_i}\otimes w_i^{\pm}][\tilde{v}_{T_i}\otimes w_i^0,\tilde{v}_{T_i}\otimes(\zeta_i^{\pm1}-1)w_i^{\pm}]}|
\notag\\\label{torTi}
&=1
\end{align}
For $T_0$, we have
$\partial_2=0,\partial_1=0$.
\begin{align*}
H_2(T_0)&=<\tilde{u}_{T_0}\otimes e_i>(i=1,2,3)
\\
H_1(T_0)&=<\tilde{m}_0\otimes e_i,\tilde{l}_0\otimes e_i>
\\
H_0(T_0)&=<\tilde{v}_{T_0}\otimes e_i>
\end{align*}
\begin{align}\label{torT0}
\tau(C_{*\rho}(T_0))=1
\end{align}
For $A_i(i=1,2,3)$, we have
$$0\longrightarrow C_{1,\rho}(A_i)\longrightarrow C_{0,\rho}(A_i)\longrightarrow0
$$
where $\partial_1=x_i^{r_i}-I$.
\\
We have
\begin{align*}
H_1(A_i)&=<\tilde{b}_i\otimes w_i^0>
\\
H_0(A_i)&=<\tilde{v}_{A_i}\otimes w_i^0>
\end{align*}
\begin{align}
\tau(C_{*,\rho}(A_i))&=|\frac{[\tilde{b}_i\otimes w_i^0,\tilde{b}_i\otimes w_i^{\pm}]}{[\tilde{v}_{A_i}\otimes(x_i^{r_i}-I) w_i^{\pm},\tilde{v}_{A_i}\otimes w_i^0]}|
\notag\\
&=|\frac{[\tilde{b}_i\otimes w_i^0,\tilde{b}_i\otimes w_i^{\pm}]}{[\tilde{v}_{A_i}\otimes(\zeta_i^{\pm r_i}-1) w_i^{\pm},\tilde{v}_{A_i}\otimes w_i^0]}|
\notag\\\label{torAi}
&=\frac{1}{|\zeta_i^{r_i}-1||\zeta_i^{-r_i}-1|}
\end{align}
For $A_0$, we have $\partial_1=0$.
\begin{align*}
H_1(A_0)&=<\tilde{b}_0\otimes e_i>(i=1,2,3)
\\
H_0(A_0)&=<\tilde{v}_{A_0}\otimes e_i>
\end{align*}
\begin{align}\label{torA0}
\tau(C_{*\rho}(A_0))=1
\end{align}
For $B$, we have
$$0\longrightarrow C_{2,\rho}(B)\stackrel{\partial_2}{\longrightarrow} C_{1,\rho}(B)\stackrel{\partial_1}{\longrightarrow} C_{0,\rho}(B)\longrightarrow0
$$
where
$$\partial_2=\begin{pmatrix}O&O&O\\O&O&O\\O&O&O\\x_1-I&x_2-I&x_3-I\end{pmatrix},\partial_1=\begin{pmatrix}x_1-I&x_2-I&x_3-I&O\end{pmatrix}
$$
We have
\begin{align*}
H_2(B)&=<\tilde{u}_{i,B}\otimes w_i^0,(\tilde{u}_{1,B}+\tilde{u}_{2,B}x_1+\tilde{u}_{3,B}x_2x_1)\otimes e_i>(i=1,2,3)
\\
H_1(B)&=<\tilde{x}_i\otimes w_i^0,(\tilde{x}_1+\tilde{x}_2x_1+\tilde{x}_3x_2x_1)\otimes e_i>
\end{align*}
Since the rank of matrix $\partial_1$ is 3, we have $H_0(B)=0$. Also, $H_3(B) = 0$ since $B$ is simple homotopy equivalent to a 2-complex.

\begin{align}
&\tau(C_{*,\rho}(B))\notag
\\
&=|[\tilde{u}_{i,B}\otimes w_i^0,\tilde{u}\otimes e_i,\tilde{u}_{1,B}\otimes w_1^{\pm},\tilde{u}_{2,B}\otimes w_2^-]^{-1}\notag
\\
&\ \ \ [\tilde{v}_B\otimes(x_1-I)w_1^{\pm},,\tilde{v}_B\otimes(x_2-I)w_2^-]^{-1}\notag
\\
&\ \ \ [\tilde{x}_i\otimes w_i^0,\tilde{x}\otimes e_i,\tilde{h}\otimes(x_1-I)w_1^{\pm},\tilde{h}\otimes(x_2-I)w_2^-,\tilde{x}_1\otimes w_1^{\pm},\tilde{x}_2\otimes w_2^-]|
\notag\\
&=|[\tilde{u}_{i,B}\otimes w_i^0,\tilde{u}\otimes e_i,\tilde{u}_{1,B}\otimes w_1^{\pm},,\tilde{u}_{2,B}\otimes w_2^-]^{-1}\notag
\\
&\ \ \ [\tilde{v}_B\otimes(\zeta_1^{\pm1}-1)w_1^{\pm},\tilde{v}_B\otimes(\zeta_2^{-1}-1)w_2^-]^{-1}\notag
\\
&\ \ \ [\tilde{x}_i\otimes w_i^0,\tilde{x}\otimes e_i,\tilde{h}\otimes(\zeta_1^{\pm1}-1)w_1^{\pm},\tilde{h}\otimes(\zeta_2^{-1}-1)w_2^-,\tilde{x}_1\otimes w_1^{\pm},\tilde{x}_2\otimes w_2^-]|
\notag\\
&=|[\tilde{u}_{i,B}\otimes w_i^0,\tilde{u}\otimes e_i,\tilde{u}_{1,B}\otimes w_1^{\pm},,\tilde{u}_{2,B}\otimes w_2^-]^{-1}[\tilde{v}_B\otimes w_1^{\pm},\tilde{v}_B\otimes w_2^-]^{-1}\notag
\\
&\ \ \ [\tilde{x}_i\otimes w_i^0,\tilde{x}\otimes e_i,\tilde{h}\otimes w_1^{\pm},\tilde{h}\otimes w_2^-,\tilde{x}_1\otimes w_1^{\pm},\tilde{x}_2\otimes w_2^-]|
\notag\\
&=1\label{torB}
\end{align}
where $\tilde{x}=\tilde{x}_1+\tilde{x}_2x_1+\tilde{x}_3x_2x_1$, $\tilde{u}=\tilde{u}_{1,B}+\tilde{u}_{2,B}x_1+\tilde{u}_{3,B}x_2x_1$.

In the long exact sequence for twisted homology group, we have isomorphisms
$$0\longrightarrow\bigoplus_{i=1}^3 H_*(T_i)\longrightarrow\bigoplus_{i=1}^3 H_*(A_i)\oplus H_*(B)\longrightarrow0
$$
Then $C_{*,\rho}(X)$ is acyclic as follows.

We have
$$0\longrightarrow\bigoplus_{i=0}^3H_0(T_i)\longrightarrow\bigoplus_{i=0}^3H_0(A_i)\longrightarrow0
$$
where $\partial(\tilde{v}_{T_i}\otimes w_i^0)=\tilde{v}_{A_i}\otimes w_i^0$, $\partial(\tilde{v}_{T_0}\otimes e_i)=\tilde{v}_{A_0}\otimes e_i$, $det(\partial)=1$.
$$0\longrightarrow\bigoplus_{i=0}^3H_1(T_i)\longrightarrow\bigoplus_{i=0}^3H_1(A_i)\oplus H_1(B)\longrightarrow0
$$
where $\partial(\tilde{m}_i\otimes w_i^0)=(\tilde{x}_i-\tilde{b}_iQ_i)\otimes w_i^0$, $\partial(\tilde{l}_i\otimes w_i^0)=\tilde{b}_iP_i\otimes w_i^0$, $\partial(\tilde{m}_0\otimes e_i)=(\tilde{x}_1+\tilde{x}_2x_1+\tilde{x}_3x_1x_2)\otimes e_i$, $\partial(\tilde{l}_0\otimes e_i)=\tilde{b}_0\otimes e_i$, $Q_i=\sum_{j=1}^{q_i}x^{-jr_i}$, $P_j=\sum_{j=0}^{p_i-1}x^{jr_i}$, $det(\partial)=p_1p_2p_3$.
$$0\longrightarrow\bigoplus_{i=0}^3H_2(T_i)\longrightarrow H_2(B)\longrightarrow0
$$
where $\partial(\tilde{u}_{T_i}\otimes w_i^0)=\tilde{u}_{i,B}\otimes w_i^0$, $\partial(\tilde{u}_0\otimes e_i)=(\tilde{u}_{1,B}+\tilde{u}_{2,B}x_1+\tilde{u}_{3,B}x_2x_1)\otimes e_i$, $det(\partial)=1$.

According to Multiplicativity lemma, Equations \ref{torTi}, \ref{torT0}, \ref{torAi}, \ref{torA0}, \ref{torB} and the calculations about homology above, we have
\begin{align*}
\Tor(C_{*,\rho}(X))=\frac{p_1p_2p_3}{\prod_{i=1}^34\sin^2\frac{2\pi r_in_i}{p_i}}
\end{align*}

\end{proof}
\subsection{Modular data from Seifert fibered spaces}\label{subsec:MTC_SFS}
We will show that the modular data constructed from 3-component SFSs are related to the Temperley-Lieb-Jones categories at root of unit. So let us collect some basic facts about those. For references, see for instance \cite{wangbook}. 

Let $A$ be a complex number such that $A^4 \neq 1$. For an integer $n$, define the quantum integer $[n]_A = \frac{A^{2n}-A^{-2n}}{A^{2}-A^{-2}}$. So $[0]_A = 0, \ [1]_A = 1, \ [2]_A = A^2+A^{-2}$. For each $A$, usually called the Kauffman variable, such that $A^4$ is a primitive $r$-th root of unity for some integer $r \geq 2$, there is an associated premodular category, called the Temperley-Lieb-Jones category and denoted by $\TLJ(A)$. The category has the label set (simple objects) $\IntSet{p-2}$ where the label $0$ is the unit object.  For $i,j \in \IntSet{p-2}$, the quantum dimension is
$$d_j(A) = (-1)^j [j+1]_A = (-1)^j \frac{A^{2j+2}-A^{-2j-2}}{A^{2}-A^{-2}},$$
the twist is
$$\theta_j(A) = (-A)^{j(j+2)},$$
and the (un-normalized) $S$-matrix is
$$\tilde{S}_{ij}(A) = (-1)^{i+j}[(i+1)(j+1)]_A.$$
The total dimension can be computed directly,
$$D(A) = \frac{\sqrt{2r}}{|A^2-A^{-2}|}.$$
Denote by $\TLJ(A)_0$ (resp. $\TLJ(A)_0$) the subcategory linearly spanned by even (resp. odd) labels. We call $\TLJ(A)_0$ and $\TLJ(A)_1$ the even and odd subcategory of $\TLJ(A)$, respectively.The even and odd subcategory has the same dimension, both equal to $\frac{D(A)}{\sqrt{2}}$.

It is well known that if $A$ is a primitive $4r$-th root of unity, then $\TLJ(A)$ is non-degenerate. If $r$ is odd and $A$ is a primitive $2r$-th root of unity, then $\TLJ(A)$ is degenerate, but the even subcategory $\TLJ(A)_0$ is non-degenerate.

Now we consider the construction of modular data.
As before, set $M = \{0;(o,0); (p_1, q_1), (p_2, q_2),  (p_3, q_3)\}$. Here each pair $(p_k,q_k)$ are co-prime. Choose integers $s_k$ and $r_k$ such that $p_ks_k - q_k r_k = 1$. If $q_k$ is odd, set $c_k = p_kq_ks_k - r_k$. Otherwise, set $c_k = p_kq_ks_k - r_k(p_k-1)^2$. Let $A_k = -\exp (\frac{2\pi i}{4p_k}c_k)$. Note that while $c_k$ depends on the choice of $s_k$ and $r_k$, $A_k$ does not. Moreover, $A_k$ is a primitive $4p_k$-th root of unity if $q_k$ is odd, a primitive $2p_k$-th root of unity if $q_k = 0 \mod 4$, and a primitive $p_k$-th root of unity if $q_k = 2 \mod 4$. In the latter two cases, $p_k$ clearly must be odd. Hence,  in all cases, $A_k^4$ is a primitive $p_k$-th root of unity.

If some $q_k'$s are even, we re-arrange the elements of $\nabchar(M)$ as follows. For $(p,q)$ co-prime, $j \in \IntSet{p-2}$, let 
\begin{equation*}
   n_{p,q}(j) = 
   \begin{cases}
       \frac{p-1-j}{2},              & q \text{ even and }j \text{ even} \\
       \frac{j+1}{2}, & \text{ otherwise} 
   \end{cases}
\end{equation*}
Then from Equation \ref{equ:R0M_original}, $\nabchar(M)$ can also be written as
\begin{equation}
\begin{split}
    &\biggl\{(n_{p_1,q_1}(j_1), n_{p_2,q_2}(j_2), n_{p_3,q_3}(j_3), \frac{1}{2}) \ | \ j_k \in \IntSet{p_k-2}^e,  k = 1,2,3\biggr\}\    \\ 
\sqcup\ &\biggl\{(n_{p_1,q_1}(j_1), n_{p_2,q_2}(j_2), n_{p_3,q_3}(j_3), 0) \ \  | \ j_k \in \IntSet{p_k-2}^o,  k = 1,2,3\biggr\}
\end{split}
\end{equation}

Thus, the elements of $\nabchar(M)$ are indexed by $\Vec{j} \in \prod_{k=1}^3 \IntSet{p_k-2}^e\  \sqcup \ \prod_{k=1}^3 \IntSet{p_k-2}^o$. Given such a $\Vec{j} = (j_1, j_2, j_3)$, denote a corresponding representation by $\rho_{\Vec{j}}$. (The choice of a representative  is irrelevant.) 

 Proposition \ref{prop:SFS_torsion} shows that all non-Abelian characters of $M$ are adjoint acyclic and Proposition \ref{prop:SFS_general_CS} shows that the CS invariants of non-Abelian characters are all rational. We choose the candidate label set $L(M)$ to be $\nabchar(M)$.

We propose the correspondence between $L(M)$ and loop operators by the following map,

\begin{align}\label{equ:irrep_to_local_general}
    \rho_{\Vec{j}} & \mapsto \bigl\{ (x_k^{c_k},\Sym^{j_k})\ |\  k = 1,2,3\bigr\}.
\end{align}
Moreover, we designate $\rho_{\Vec{0}} = \rho_{(0,0,0)}$ as the unit object, which of course corresponds to the loop operator
\begin{align}\label{equ:unit_general}
    \unitobj = \rho_{\Vec{0}} &\mapsto \bigl\{ (x_k^{c_k},\Sym^{0})\ |\  k = 1,2,3\bigr\}.
\end{align}

The following two lemmas are direct consequences of Proposition \ref{prop:SFS_general_CS} and Proposition \ref{prop:SFS_torsion}, respectively.
\begin{lemma}\label{lem:CS_general}
Let $M, \ c_k, \ A_k$ be given as above. For each $\Vec{j} = (j_1, j_2, j_3) \in \prod_{k=1}^3 \IntSet{p_k-2}^e\  \sqcup \ \prod_{k=1}^3 \IntSet{p_k-2}^o$ with $\rho_{\Vec{j}}$ a corresponding representation,  then
 \begin{align}
     \CS(\rho_{\Vec{j}}) = \sum_{k=1}^3 \frac{-c_k}{4p_k} (j_k+1)^2.
 \end{align}
 As a consequence, 
 \begin{align}
     e^{-2\pi i \CS(\rho_{\Vec{j}})} = \prod_{k=1}^3 (-A_k)^{(j_k+1)^2} =(-A_1A_2A_3) \prod_{k=1}^3 \theta_{j_k}(A_k).
 \end{align}
 \begin{proof}
Note that for  $\Vec{j} = (j_1, j_2, j_3)$, $\Tr(\rho_{(j_1,j_2,j_3)}(x_i))=2\cos\frac{2\pi n_{p_i,q_i}(j_i)}{p_i}$. The formula above then follows from Proposition \ref{prop:SFS_general_CS}.
 \end{proof}
\end{lemma}

\begin{lemma}\label{lem:Tor_general}
 Let $M, \ c_k, \ A_k$ be given as above and let $D = D(A_1)D(A_2)D(A_3)/2$. For each $\Vec{j} = (j_1, j_2, j_3) \in \prod_{k=1}^3 \IntSet{p_k-2}^e\  \sqcup \ \prod_{k=1}^3 \IntSet{p_k-2}^o$ with $\rho_{\Vec{j}}$ a corresponding representation, then
 \begin{align}
     \Tor(\rho_{\Vec{j}}) &= \prod_{k = 1}^3 \ \frac{p_k}{4 \sin^2(\frac{\pi r_k (j_k+1)}{p_k})},
 \end{align}
 and hence,
 \begin{equation}
     \bigl(2 \Tor(\rho_{\Vec{j}})\bigr)^{-\frac{1}{2}} = 2 \prod_{k=1}^3 \left|\frac{d_{j_k}(A_k)}{D(A_k)}\right| \ =\  \frac{|\prod_{k=1}^3 d_{j_k}(A_k)|}{D}.
 \end{equation}
\end{lemma}

The main result of the section is the following theorem.
\begin{theorem}\label{thm:construct_premodular_general}
Let $M = \{0; (p_1, q_1), (p_2, q_2),  (p_3, q_3)\}$ and $\{A_k\}_{k = 1,2,3}$ be given as above. With the  operators and tensor unit given in Equations \ref{equ:irrep_to_local_general} and \ref{equ:unit_general}, respectively, the modular data constructed from $M$ matches that of the following pre-modular category, 
\begin{align*}
    \mathcal{B} := \left(\ \boxtimes_{k=1}^3\TLJ(A_k)_0\ \right) \bigoplus  \left(\ \boxtimes_{k=1}^3\TLJ(A_k)_1\ \right)
\end{align*}
\begin{proof}
Since $A_k^4$ is a primitive $p_k$-th root of unity, the label set for $\mathcal{B}$ is clearly $L:= \prod_{k=1}^3 \IntSet{p_k-2}^e\  \sqcup \ \prod_{k=1}^3 \IntSet{p_k-2}^o$, the same index set for $L(M)$. The modular data of $\mathcal{B}$ can be easily expressed in terms of that of the individual $\TLJ(A_k)$. For $\Vec{i}, \Vec{j} \in L$, 
$$d_{\Vec{j}} = \prod_{k=1}^3 d_{j_k}(A_k), \quad \theta_{\Vec{j}} =  \prod_{k=1}^3 \theta_{j_k}(A_k), \quad \tilde{S}_{\Vec{i}\,\Vec{j}} = \prod_{k=1}^3 \tilde{S}_{i_kj_k}(A_k).$$
Also, the total dimension of $\mathcal{B}$ is $D = D(A_1)D(A_2)D(A_3)/2$.

Lemma \ref{lem:CS_general} shows that, up to a global phase, the Chern-Simons invariant gives the twist,
$$e^{-2\pi i \CS(\rho_{\Vec{j}})} = \theta_{\Vec{j}},$$
and Lemma \ref{lem:Tor_general} shows that the torsion matches the absolute value of the normalized quantum dimension,
$$\bigl(2 \Tor(\rho_{\Vec{j}})\bigr)^{-\frac{1}{2}} = \frac{d_{\Vec{j}}}{D}.$$

Lastly, We check the $S$-matrix computed from loop operators. Given $\Vec{i} = (i_1,i_2,i_3), \ \Vec{j} = (j_1,j_2,j_3) \in L$, we have (choosing $\epsilon = -1$)
\begin{align*}
    W_{\Vec{i}}(\Vec{j}) &= \prod_{k=1}^3 \Tr_{\Sym^{j_k}}(-\rho_{\Vec{i}}(x_k^{c_k})).
\end{align*}
Note that,
\begin{align*}
    \Tr\bigl(\rho_{\Vec{i}}(x_k^{c_k})\bigr) \ = \ 2 \cos \frac{2 n_{p_k,q_k}(i_k) \pi c_k }{p_k} \ = \ 2 \cos \frac{(i_k+1) \pi c_k }{p_k}, 
\end{align*}
where the second equality holds irrelevant of the parity of $q_k$. Combining the previous two equations, we get
\begin{align*}
    W_{\Vec{i}}(\Vec{j}) \ =\  \prod_{k=1}^3 \Delta_{j_k}(-2 \cos \frac{(i_k+1)\pi c_k}{p_k}) \ = \  \prod_{k=1}^3 (-1)^{j_k} \frac{\sin \frac{(i_k+1)(j_k+1)\pi c_k}{p_k}}{\sin \frac{(i_k+1)\pi c_k}{p_k}},
\end{align*}
where $\Delta_{j_k}(\cdot)$ is the Chebyshev polynomial (see Equation \ref{equ:Delta}).
Therefore, the $(\Vec{j},\Vec{i})$-entry of the potential un-normalized $S$ matrix is,
\begin{align*}
     W_{\Vec{i}}(\Vec{j}) W_{\Vec{0}}(\Vec{i})  \ &=\ \prod_{k=1}^3 (-1)^{i_k+j_k} \frac{\sin \frac{(i_k+1)(j_k+1)\pi c_k}{p_k}}{\sin \frac{\pi c_k}{p_k}} \\ 
    &=\ \prod_{k=1}^3 \tilde{S}(A_k)_{j_k i_k},
\end{align*}
which is precisely $\tilde{S}_{\Vec{j}\,\Vec{i}}$ of $\mathcal{B}$.
\end{proof}
\end{theorem}

The premodular category produced in the previous theorem may not be modular in general, and it depends crucially on the topology of the three manifold. For a three-component SFS $M$, it is a $\bbZ_2$ homology sphere, i.e., $H^1(M,\bbZ_2) = 0$, if and only if
\begin{align*}
    p_1p_2p_3(\frac{q_1}{p_1} + \frac{q_2}{p_2} + \frac{q_3}{p_3}) \in 2\bbZ +1
\end{align*}
\begin{lemma}\label{lemma:mtc}
Assume that $r$ is odd. Suppose that
$$T(p,j,l,*)\ =\ \sum_{m \in [p]^*}\ \left(e^{(j+l)mr\frac{\pi}{p}i}-e^{(j-l)mr\frac{\pi}{p}i}-e^{(-j+l)mr\frac{\pi}{p}i}+e^{(-j-l)mr\frac{\pi}{p}i}\right)
$$
where $*=1,0$, and $[p]^*$ denotes the set of odd integers from 1 to $p-1$ if $*$ is 1 and the set of even integers in the same range otherwise.
\\
When $p$ is odd, $j\neq l$, $j+l$ is odd,
$$T(p,j,l,*)=\left\{\begin{aligned}
0&&j+l\neq p
\\
(-1)^*p&&j+l=p
\end{aligned}
\right.
$$
When $p$ is odd, $j\neq l$, $j+l$ is even,
$$T(p,j,l,*)=0
$$
When $p$ is odd, $j=l$,
$$T(p,j,l,*)=-p
$$
When $p$ is even, $j\neq l$, $j+l$ is odd,
$$T(p,j,l,*)=0
$$
When $p$ is even, $j\neq l$, $j+l$ is even,
$$T(p,j,l,*)=\left\{\begin{aligned}
0&&j+l\neq p
\\
(-1)^*p&&j+l=p
\end{aligned}
\right.
$$
When $p$ is even, $j=l$,
$$T(p,j,l,0)=\left\{\begin{aligned}
-p&&j+l\neq p
\\
0&&j+l=p
\end{aligned}
\right.
$$
$$T(p,j,l,1)=\left\{\begin{aligned}
-p&&j+l\neq p
\\
-2p&&j+l=p
\end{aligned}
\right.
$$
\end{lemma}
\begin{proof}
We prove the lemma by direct computation.
\\
When $p$ is odd, $j\neq l$, $j+l$ is odd,
\begin{align*}
T(p,j,l,1)&=\sum_{m=1,m\text{ odd}}^{p-2}(e^{(j+l)mr\frac{\pi}{p}i}-e^{(j-l)mr\frac{\pi}{p}i}+e^{(j-l)(p-m)r\frac{\pi}{p}i}-e^{(j+l)(p-m)r\frac{\pi}{p}i})
\\
&=\sum_{m=1,m\text{ odd}}^{p-2}(e^{(j+l)mr\frac{\pi}{p}i}-e^{(j-l)mr\frac{\pi}{p}i})+\sum_{m=2,\text{ even}}^{p-1}(e^{(j-l)mr\frac{\pi}{p}i}-e^{(j+l)mr\frac{\pi}{p}i})
\\
&=-\sum_{m=1}^{p-1}(-e^{(j+l)r\frac{\pi}{p}i})^m+\sum_{m=1}^{p-1}(-e^{(j-l)r\frac{\pi}{p}i})^m
\\
&=\left\{\begin{aligned}
0&&j+l\neq p
\\
-p&&j+l=p
\end{aligned}
\right.
\\
&=-T(p,j,l,0)
\end{align*}
Similarly, we get other cases.
\end{proof}
\begin{proposition}\label{prop:modular}
Given a three-component SFS $M$, the premodular category $\mathcal{B}_M$ produced in Theorem \ref{thm:construct_premodular_general} is modular if and only if $M$ is a $\bbZ_2$ homology sphere.
\begin{proof}

Since the structure from Section \ref{sec:MTC_SFS} respects the change of parametrization of Seifert fiber space, it suffices to verify the following 5 cases for $(\frac{p_1}{q_1},\frac{p_2}{q_2},\frac{p_3}{q_3})$.
$$(\frac{odd}{odd},\frac{odd}{odd},\frac{odd}{odd}),(\frac{odd}{odd},\frac{odd}{odd},\frac{even}{odd}),(\frac{odd}{odd},\frac{even}{odd},\frac{even}{odd}),
$$
$$(\frac{even}{odd},\frac{even}{odd},\frac{even}{odd}),(\frac{odd}{odd},\frac{odd}{odd},\frac{odd}{even})
$$
The first two cases correspond to $\mathbb{Z}_2$-homology sphere.  In the following, we will explicitly calculate $S^2$, which directly implies the proposition.

When $q_1,q_2,q_3$ are odd, $j_1=j_2=j_3 \mod 2$, $l_1=l_2=l_3 \mod 2$.
\\
Up to a scalar,
$$S_{(j_1,j_2,j_3),(l_1,l_2,l_3)}=(-1)^{j_1+l_1}\prod_{k=1}^3\sin j_kl_kr_k\frac{\pi}{p_k}
$$
\begin{align*}
&{}\quad(S^2)_{(j_1,j_2,j_3),(l_1,l_2,l_3)}\\
&=\sum_{(m_1,m_2,m_3)}(-1)^{j_1+m_1+m_1+l_1}\prod_{k=1}^3\sin j_km_kr_k\frac{\pi}{p_k}\sin m_kl_kr_k\frac{\pi}{p_k}
\\
&=(-1)^{j_1+l_1}\sum_{(m_1,m_2,m_3)}\prod_{k=1}^3-\frac{1}{4}(e^{(j_k+l_k)m_kr_k\frac{\pi}{p_k}i}-e^{(j_k-l_k)m_kr_k\frac{\pi}{p_k}i}-e^{(-j_k+l_k)m_kr_k\frac{\pi}{p_k}i}
\\
&{}\quad+e^{(-j_k-l_k)m_kr_k\frac{\pi}{p_k}i})
\\
&=(-1)^{j_1+l_1}(\sum_{(m_1,m_2,m_3),m_i\text{ odd}}+\sum_{(m_1,m_2,m_3),m_i\text{ even}})...
\\
&=(-1)^{j_1+l_1}(\prod_{k=1}^3T(p_k,j_k,l_k,1)+\prod_{k=1}^3T(p_k,j_k,l_k,0))
\end{align*}

When $p_1,p_2,p_3$ are odd,
\begin{align*}
(S^2)_{(j_1,j_2,j_3),(l_1,l_2,l_3)}&=\left\{\begin{aligned}
0&&(j_1,j_2,j_3)\neq(l_1,l_2,l_3)
\\
\frac{p_1p_2p_3}{32}&&(j_1,j_2,j_3)=(l_1,l_2,l_3)
\end{aligned}
\right.
\end{align*}
When $p_1,p_2$ are odd, $p_3$ is even,
\begin{align*}
(S^2)_{(j_1,j_2,j_3),(l_1,l_2,l_3)}&=\left\{\begin{aligned}
0&&(j_1,j_2,j_3)\neq(l_1,l_2,l_3)
\\
\frac{p_1p_2p_3}{32}&&(j_1,j_2,j_3)=(l_1,l_2,l_3)
\end{aligned}
\right.
\end{align*}
Thus $S^2=cI$ for the above two cases.
\\
When $p_1$ is odd, $p_2,p_3$ are even,
$$(S^2)_{(1,1,1),(l_1,l_2,l_3)}=\left\{\begin{aligned}
\frac{p_1p_2p_3}{32}&&(l_1,l_2,l_3)=(1,1,1),(1,p_2-1,p_3-1)
\\
0&&\text{otherwise}
\end{aligned}
\right.
$$
$$(S^2)_{(1,p_2-1,p_3-1),(l_1,l_2,l_3)}=\left\{\begin{aligned}
\frac{p_1p_2p_3}{32}&&(l_1,l_2,l_3)=(1,1,1),(1,p_2-1,p_3-1)
\\
0&&\text{otherwise}
\end{aligned}
\right.
$$
$$(S^2)_{(1,1,1)}=(S^2)_{(1,p_2-1,p_3-1)}
$$
When $p_1,p_2,p_3$ are even,
$$(S^2)_{(1,1,1),(l_1,l_2,l_3)}=\left\{\begin{aligned}
\frac{p_1p_2p_3}{32}&&(l_1,l_2,l_3)=(1,1,1),(1,p_2-1,p_3-1)
\\
0&&\text{otherwise}
\end{aligned}
\right.
$$
$$(S^2)_{(1,p_2-1,p_3-1),(l_1,l_2,l_3)}=\left\{\begin{aligned}
\frac{p_1p_2p_3}{32}&&(l_1,l_2,l_3)=(1,1,1),(1,p_2-1,p_3-1)
\\
0&&\text{otherwise}
\end{aligned}
\right.
$$
$$(S^2)_{(1,1,1)}=(S^2)_{(1,p_2-1,p_3-1)}
$$
$S^2$ is degenerate for above two cases.
\\
When $q_1,q_2$ are odd, $q_3$ is even, $j_1=j_2 \mod 2$, $l_1=l_2 \mod 2$, $j_3=0 \mod 2$, $l_3=0 \mod 2$.
\begin{align*}
(S^2)_{(j_1,j_2,j_3),(l_1,l_2,l_3)}&=\prod_{k=1}^2T(p_k,j_k,l_k,1)T(p_3,j_3,l_3,0)+\prod_{k=1}^3T(p_k,j_k,l_k,0)
\end{align*}
When $p_1,p_2,p_3$ are odd,
$$(S^2)_{(1,1,2),(l_1,l_2,l_3)}=\left\{\begin{aligned}
-\frac{p_1p_2p_3}{32}&&(l_1,l_2,l_3)=(1,1,2),(p_1-1,p_2-1,2)
\\
0&&\text{otherwise}
\end{aligned}
\right.
$$
$$(S^2)_{(p_1-1,p_2-1,2),(l_1,l_2,l_3)}=\left\{\begin{aligned}
-\frac{p_1p_2p_3}{32}&&(l_1,l_2,l_3)=(1,1,2),(p_1-1,p_2-1,2)
\\
0&&\text{otherwise}
\end{aligned}
\right.
$$
$S^2$ is degenerate.
\end{proof}
\end{proposition}

It is worth noting even if every $\TLJ(A_k)$ appearing in the construction of $\mathcal{B}_M$ in Theorem \ref{thm:construct_premodular_general} is not modular, $\mathcal{B}_M$ could still be modular. For instance, for the SFS $M_0 = (0;(o,0); (5,1), (3,2), (5,4))$, the corresponding Kauffman variables are $A_1 = -e^{\frac{i \pi }{10}},\ A_2 = -e^{\frac{i \pi }{3}},\ A_3 = -e^{\frac{2 i \pi }{5}}$. It is direct to see that $\TLJ(A_1)$ is modular, but  $\TLJ(A_2)$ and $\TLJ(A_3)$ are not. However, $M_0$ is a $\bbZ_2$ homology sphere, by Proposition \ref{prop:modular}, $\mathcal{B}_{M_0}$ is modular, a rank-8 MTC.

\subsection{Examples: Realization of $\SU(2)_k$}\label{subsec:33r}
Here we study a special class of SFSs with three components, namely, $M(r):= \{0;(o,0); (3,1), (3,1), (r,1)\}$. We show explicitly that different choice of characters as the unit object may lead to different theories. In fact, it will be proved that from $M(r)$ we can construct either the MTC $\SU(2)_{r-2}$ or $\TLJ(e^{\frac{2\pi i}{4r}})$.

For each integer $r \geq 2$, there is a unitary MTC, usually denoted by $\SU(2)_{r-2}$ \cite{bonderson2007non}, which is closely related to the Temperley-Lieb-Jones categories. Here $r-2$ is called the level of the MTC. It has the same label set as $\TLJ(e^{\frac{2\pi i}{4r}})$, but differs from it in modular data by some signs. Explicitly, setting $A = e^{\frac{2\pi i}{4r}}$, the modular data for $\SU(2)_{r-2}$ is given as follows,
$$\theta_j\  =\ A^{j(j+2)}\ =\ e^{\frac{2\pi i\, j(j+2)}{4r}},$$
$$\tilde{S}_{ij}\ =\ [(i+1)(j+1)]_{A}\ = \ \frac{\sin \frac{(i+1)(j+1)\pi}{r}}{\sin \frac{\pi}{r}} .$$
In particular, its quantum dimensions are all positive (since it is unitary),
$$d_j \ =\  [j+1]_{A} \ = \ \frac{\sin \frac{(j+1)\pi}{r}}{\sin \frac{\pi}{r}},$$
and the total dimension is 
$$D = \sqrt{\frac{r}{2}}\frac{1}{\sin \frac{\pi}{r}}.$$
Note that $d_j = |d_j(A)|$ and $D = D(A)$, where $d_j(A)$ and $D(A)$ are the quantum dimension of $j$ and total dimension of $\TLJ(A)$, respectively.

We will use notations from Section \ref{subsec:variety} and \ref{subsec:MTC_SFS}. The non-Abelian characters of $M(r)$ is given by
\begin{equation}\label{equ:chiMr}
\begin{split}
    \nabchar(M(r)) = &\left\{\left(\frac{1}{2}, \frac{1}{2}, \frac{j+1}{2}, \frac{1}{2}\right)\ | \ (0,0,j) \in \{0\} \times \{0\} \times \IntSet{r-2}^{e}\right\}  \\
            \sqcup &\left\{\left(1, 1, \frac{j+1}{2}, 0\right)\ | \ (1,1,j) \in \{1\} \times \{1\} \times \IntSet{r-2}^{o}\right\}.
\end{split}
\end{equation}

Thus, each $j \in \IntSet{r-2}$ corresponds to a non-Abelian character indexed by $(j \bmod 2, j \bmod 2, j)$. We denote the corresponding representation by $\rho_{j}$ (instead of using the triple as the subscript). The eigenvalues of $\rho_j(x_3)$ are $e^{\pm \frac{(j+1)\pi i}{r}}$. The eigenvalues of $\rho_j(x_1)$ and those of $\rho_j(x_2)$ are both $e^{\pm \frac{a_j\pi i}{3}}$, where $a_j = 1$ if $j$ even and $a_j = 2$ otherwise.

Also, it is direct to see that $c_1 = c_2 = c_3 = 1$, and $A_1 = A_2 = -e^{\frac{\pi i }{6}}$, $A_3 = -e^{\frac{2\pi i}{4r}}$.

In Section \ref{subsec:MTC_SFS}, we chose the candidate label set $L(M(r))$ to be $\nabchar(M(r))$, and defined the following map from $\nabchar(M(r))$ to loop operators,
\begin{align}\label{equ:irrep_to_local_33r}
    \rho_j = & \mapsto \bigl\{ (x_1,\Sym^{j \bmod 2}), (x_2,\Sym^{j \bmod 2}), (x_3,\Sym^{j})\bigr\}.
\end{align}
It can be checked directly that for $i,j \in \IntSet{r-2}$, $\Tr(\rho_i(x_1)) = \Tr(\rho_i(x_2)) = \pm 1$, and it follows that,
\begin{align*}
    W_{i}(j) &= \Tr_{\Sym^{j \bmod 2}}(-\rho_i(x_1))\,\Tr_{\Sym^{j \bmod 2}}(-\rho_i(x_2))\,\Tr_{\Sym^{j}}(-\rho_i(x_3)) \\
    &=\Tr_{\Sym^{j}}(-\rho_i(x_3)).
\end{align*}
Hence, we may as well choose a simplified map to loop operators,
\begin{equation}\label{equ:irrep_to_local_33r_simple}
    \rho_j \mapsto \{(x_3, \Sym^j)\}.
\end{equation}
The unit object was chosen to be $\rho_0$ which corresponds to the loop operator $(x_3, \Sym^0)$.
By Theorem \ref{thm:construct_premodular_general}, the modular data match that of the premodular category,
\begin{align}
    \mathcal{B}_{M(r)} = \left(\ \boxtimes_{k=1}^3\TLJ(A_k)_0\ \right) \bigoplus  \left(\ \boxtimes_{k=1}^3\TLJ(A_k)_1\ \right).
\end{align}

Note that $\TLJ(A_1) = \TLJ(-e^{\frac{\pi i }{6}})$ has label set $\{0,1\}$, the twists $\theta_0 = 1$, $\theta_1 = i$, and un-normalized $S$-matrix,
\begin{equation*}
\tilde{S} = 
\begin{pmatrix}
1 & -1 \\
-1 & -1
\end{pmatrix}.
\end{equation*}
This means that $\mathcal{B}_{M(r)}$ has the  same twists for even labels and $S$-matrix   as $\TLJ(A_3)$. The twists for odd labels differ by a minus  sign between the two theories. Let $A(r) = -A_3 = e^{\frac{2\pi i}{4r}}$. Note that a change of the Kauffman variable from $A$ to $-A$ does not change the $S$-matrix. It follows that $\mathcal{B}_{M(r)}$ and $\TLJ(A(r))$ has the same modular data. In fact, they are isomorphic.

Therefore, by using the loop operator correspondence in Equation \ref{equ:irrep_to_local_33r_simple} and letting $\rho_0$ be the unit object, we recover the MTC $\TLJ(A(r))$.

Now we examine an alternative choice of the unit object. Since $M(r)$ is a $\bbZ_2$ homology sphere, a potential unit object $\rho_{\alpha_0}$ can be determined by the equation,
\begin{align}
    \left|\sum_{\rho \in \nabchar(M(r))} \frac{\exp(-2\pi i \CS(\rho))}{2 \Tor(\rho)}\right| = (2 \Tor(\rho_{\alpha_0}))^{-\frac{1}{2}}.
\end{align}
Such a $\rho_{\alpha_0}$ would have quantum dimension in absolute value equal to 1 in any MTC produced by $M(r)$. Since we already know that we can produce  $\TLJ(A(r))$ from $M(r)$ and the only non-unit object in $\TLJ(A(r))$ whose quantum dimension is 1 in absolute value is $\rho_{r-2}$, we can choose $\rho_{r-2}$ as the unit object in a new theory.

In this case, we reverse the previous order of the simple objects. Denote by $\tilde{\rho}_j := \rho_{r-2-j}$, $j \in \IntSet{r-2}$. Set $\tilde{\rho}_0 = \rho_{r-2}$ as the unit object. The correspondence between characters and loop operators is now defined as,
\begin{equation}\label{equ:irrep_to_local_33r_V2}
    \tilde{\rho}_j \mapsto (x_3, \Sym^j).
\end{equation}
We claim that with above choice of unit object and loop operators, the modular data produced from $M(r)$ matches that of $\SU(2)_{r-2}$ where $\tilde{\rho}_j$ corresponds to $j$ in the label set of $\SU(2)_{r-2}$. See Section \ref{subsec:MTC_SFS} for a collection of facts about $\SU(2)_{r-2}$.

Firstly, by Lemma \ref{lem:CS_general}, up to an irrelevant phase factor, 
\begin{equation}\label{equ:CSrho_j}
    \CS(\rho_j) = -\frac{j(j+2)}{4r} + \frac{1-(-1)^j}{4} \mod 1.
\end{equation}
Then rewriting above equation in terms of $\tilde{\rho}_j$, we get, again up to an irrelevant factor,
\begin{equation}\label{equ:CSrho_jtilde}
    \CS(\tilde{\rho}_j) = -\frac{j(j+2)}{4r}  \mod 1.
\end{equation}
Thus, 
\begin{equation}
    e^{-2\pi i \CS(\tilde{\rho}_j)} \ =\  e^{\frac{2\pi i \, j(j+2)}{4r}} 
\end{equation}
is the twist $\theta_j$ of $\SU(2)_{r-2}$.

Next, we check the $S$-matrix.
\begin{equation}
    W_0(j) \ =\  \Tr_{\Sym^j}(-\tilde{\rho}_0(x_3))\  =\  \Delta_j(2 \cos \frac{\pi}{r}) \ = \ \frac{\sin \frac{(j+1)\pi}{r}}{\sin \frac{\pi}{r}},
\end{equation}
and the $(j,i)$-entry of the potential $S$-matrix is,
\begin{align}
W_{i}(j)W_0(i)  &= \Tr_{\Sym^j}(-\tilde{\rho}_i(x_3))\, W_0(i) \\
           &= \Delta_j(2 \cos \frac{(i+1)\pi}{r}) \Delta_i(2\cos \frac{\pi}{r})  \\
           &=  \frac{\sin \frac{(i+1)(j+1)\pi}{r}}{\sin \frac{\pi}{r}},
\end{align}
which is $\tilde{S}_{ji}$ of $\SU(2)_{r-2}$.

Lastly, by Lemma \ref{lem:Tor_general},
\begin{align}
    \bigl(2 \Tor(\tilde{\rho}_{j})\bigr)^{-\frac{1}{2}}\ =\  \bigl(2 \Tor(\rho_{r-2-j})\bigr)^{-\frac{1}{2}} \ =\  \frac{|d_{r-2-j}(A_3)|}{D(A_3)},
\end{align}
where we used the fact that in $\TLJ(A_1) = \TLJ(A_2)$, the two simple objects have quantum dimensions $\pm 1$ and thus the dimension of the category is $D(A_1) = \sqrt{2}$. Also note that $A_3 = -e^{\frac{2\pi i}{4r}}$, then $|d_{r-2-j}(A_3)| = |d_{j}(A_3)|$ and $D(A_3)$ are equal to the quantum dimension $d_j$ and the total dimension $D$, respectively, in $\SU(2)_{r-2}$. Hence, the torsion invariant computes the normalized quantum dimension,  
\begin{align}
    \bigl(2 \Tor(\tilde{\rho}_{j})\bigr)^{-\frac{1}{2}}\ =\  \frac{d_j}{D}.
\end{align}

To summarize, for the SFS $M(r)$, two choices of the unit object together with appropriate definition of loop operators produce the MTCs  $\TLJ(e^{\frac{2\pi i}{4r}})$ and $\SU(2)_{r-2}$, with the former non-unitary and the latter unitary.

\subsection{Graded product of graded premodular categories}
In Section \ref{subsec:MTC_SFS}, we have seen that the premoduar category resulting from three-component SFSs is formed from three Temperley-Lieb-Jones categories, by taking the Deligne product of the even sectors, that of the odd sectors, and suming them up. Here we generalize the operation.

\begin{definition}\label{def:homogeneous_product}
Let $\C = \oplus_{g \in G} \C_g$ and $\D = \oplus_{g \in G} \D_g$ be two $G$-graded premodular tensor categories for some finite group $G$ (which must be Abelian). The graded product of $\C$ and $\D$ is again a $G$-graded premodular category $\C \gtensor \D = \oplus_{g \in G} (\C \gtensor \D)_g$ such that $(\C \gtensor \D)_g:= \C_g \boxtimes \D_g$.
\end{definition}

The monoidal and braiding structure on $\C \gtensor \D$ is defined in the obvious way which make it into a premodular category. Another way to see this is that $\C \gtensor \D$ is a full subcategory of the premodular category $\C \boxtimes \D$ and is closed under tensor product and braiding. The graded product operation $\gtensor$ is associative up to canonical equivalence.

For a Kauffman variable $A$, $\TLJ(A)$ is a $\bbZ_2$-graded premodular category with $\TLJ(A)_0$ spanned by even labels and  $\TLJ(A)_1$ odd labels.
Hence, Theorem \ref{thm:construct_premodular_general} states that, for a three-component SFS $M = \{0;(o,0); (p_1, q_1), (p_2, q_2),  (p_3, q_3)\}$ with $A_k, k = 1,2,3$ defined as in Section \ref{subsec:MTC_SFS}, the premodular category resulting from $M$ is $\mathcal{B}_M = \TLJ(A_1) \gtensor \TLJ(A_2) \gtensor \TLJ(A_3)$.

The graded product operation provides method to construct new premodular categories from old ones. A very interesting question is when the graded product of two pre-modular categories is modular. For instance, take $A_1 = -e^{\frac{i \pi }{6}},\ A_2 = -e^{-\frac{i \pi }{5}}$. Here $A_1$ is a primitive $12$-th root of unity and $A_2$ a primitive $5$-th root of unity. Hence $\TLJ(A_1)$ is modular of rank $2$ and $\TLJ(A_2)$ is none modular of rank $4$. Their $S$-matrices are given by,
\begin{equation}
    \tilde{S}(A_1) = 
\left(
\begin{array}{cc}
 1 & -1 \\
 -1 & -1 \\
\end{array}
\right),
\quad
\tilde{S}(A_2) = 
\left(
\begin{array}{cccc}
 1 & \varphi & \varphi & 1 \\
 \varphi & -1 & -1 & \varphi \\
 \varphi & -1 & -1 & \varphi \\
 1 & \varphi & \varphi & 1 \\
\end{array}
\right),
\end{equation}
where $\varphi = \frac{1}{2} (1-\sqrt{5})$. Then the $S$-matrix of $\TLJ(A_1) \gtensor \TLJ(A_2)$ with its simple objects ordered as $\{0 \boxtimes 0, 0 \boxtimes 2, 1 \boxtimes 1, 1 \boxtimes 3\}$ is,
\begin{equation}
    \tilde{S} = 
\left(
\begin{array}{cccc}
 1 & \varphi & -\varphi & -1 \\
 \varphi & -1 & 1 & -\varphi \\
 -\varphi & 1 & 1 & -\varphi \\
 -1 & -\varphi &-\varphi & -1 \\
\end{array}
\right),
\end{equation}
which can be checked straightforwardly to be non-degenerate. Thus $\TLJ(A_1) \gtensor \TLJ(A_2)$ is modular.

We leave the question of when the graded product of two arbitrary graded (and more generally multiple) premodular categories is modular as a future direction. In the rest of this section, we focus on the case where the group is $\bbZ_2$ and study a special class of $\bbZ_2$-graded modular categories, namely $\SU(2)_k$. For basic facts, see Section \ref{subsec:33r}.

Let $\C = \C_0 \oplus \C_1$ be a $\mathbb{Z}_2$-graded MTC. Denote by $I$ the label set of $\C$ and partition $I = I_{0} \sqcup I_{1}$ where $I_{\alpha}$ consists of objects of $I$ that are in the $\C_{\alpha}$ sector. To avoid confusion, when there is more than one MTC present, we write $I(\C), \ \tilde{S}(\C)$, etc.
\begin{proposition}
\label{prop:two_product}
Let $\C$ and $\D$ be two $\mathbb{Z}_2$-graded MTCs. Then $\C \gtensor \D$ is a proper (i.e., degenerate) premodular category if and only if there exist $i \in I(\C)$, $j \in I(\D)$, scalars $c_0(\C),\ c_1(\C),\ c_0(\D),$ and $c_1(\D) $, such that,
\begin{enumerate}
    \item $i$ and $j$ belong to sectors of the same parity;
    \item the following equations concerning $S$-entries hold:
    \begin{equation*}
        \tilde{S}(\C)_{ik} = 
        \begin{cases}
            c_0(\C) d_k(\C) & k \in I_0(\C)\\
            c_1(\C) d_k(\C) & k \in I_1(\C)
        \end{cases}
        \qquad
        \tilde{S}(\D)_{jk} = 
        \begin{cases}
            c_0(\D) d_k(\D) & k \in I_0(\D)\\
            c_1(\D) d_k(\D) & k \in I_1(\D)
        \end{cases}
    \end{equation*}
    \item $c_0(\C)/c_1(\C) = c_1(\D)/c_0(\D) \neq 1$.
\end{enumerate}
\begin{proof}
The main idea is to show that the conditions presented in the statement of the proposition are equivalent to the property that in the $S$-matrix of $\C \gtensor \D$, the row corresponding to the object $i \boxtimes j$ is proportional to the first row (i.e., the row corresponding to the unit object).
\end{proof}
\end{proposition}
\begin{remark}
In the above proposition, the conditions $c_0(\C)/c_1(\C) \neq 1 $ and $ c_1(\D)/c_0(\D) \neq 1$ are used to eliminate the trivial case where $i$ and $j$ are both the unit object. When neither of $i$ nor $j$ is the unit object, those conditions automatically hold since otherwise the $S$-matrix of $\C$ or $\D$ would be degenerate. Also, note that if either $\C_0$ or $\D_0$ is non-degenerate, then $i$ and $j$ must be in the sector of odd parity. 
\end{remark}

For $m \geq 0$, $\SU(2)_m$ is a $\bbZ_2$-graded MTC with  $\left(\SU(2)_m \right)_0$ spanned by even labels and $\left(\SU(2)_m \right)_1$ by odd labels.
\begin{theorem}\label{thm:SU2tensor}
For $m,n \geq 0$, $\SU(2)_m \gtensor \SU(2)_n$ is an MTC if and only if the pair $(m,n)$ have different parity. In particular, $\SU(2)_m \gtensor \SU(2)_m$ is always degenerate. 
\begin{proof}
In $\SU(2)_m$, the un-normalized $S$-matrix is given by,
\begin{align*}
    \tilde{S}_{ab} = \frac{\sin \frac{(a+1)(b+1)\pi}{m+2}}{\sin \frac{\pi}{m+2}}.
\end{align*}
Hence, $\tilde{S}_{mb} = (-1)^b \tilde{S}_{0b} = (-1)^b d_b$. For $(m,n)$ with the same parity, with the notation from the statement of Proposition \ref{prop:two_product}, we choose $i = m, j = n$. Then the relevant constants are $c_0(\SU(2)_m) = c_0(\SU(2)_n) = 1$, $c_1(\SU(2)_m) = c_1(\SU(2)_n) = -1$ which satisfies the conditions stated in that proposition, and hence $\SU(2)_m \gtensor \SU(2)_n$ is degenerate. For the converse direction, it can be seen that the only non-unit simple object in $\SU(2)_m$ for which $c_0(\SU(2)_m)$ and $c_1(\SU(2)_m)$ exist is the object $m$. Therefore, if $(m,n)$ have different parity, the only pair of indexes for $(i,j)$ is $(m,n)$ which contradicts the first condition of Proposition \ref{prop:two_product}. This implies that $\SU(2)_m \gtensor \SU(2)_n$ is non-degenerate.
\end{proof}
\end{theorem}
\begin{example}
By Theorem \ref{thm:SU2tensor}, $\SU(2)_2 \gtensor \SU(2)_3$ is an MTC of rank $6$. Its un-normalized $S$-matrix and $T$-matrix are given by,
\begin{equation*}
  \tilde{S}=  \left(
\begin{array}{cccccc}
 1 & \frac{1}{2} \left(1+\sqrt{5}\right) & 1 & \frac{1}{2} \left(1+\sqrt{5}\right) & \frac{1+\sqrt{5}}{\sqrt{2}} & \sqrt{2}
   \\
 \frac{1}{2} \left(1+\sqrt{5}\right) & -1 & \frac{1}{2} \left(1+\sqrt{5}\right) & -1 & -\sqrt{2} &
   \frac{1+\sqrt{5}}{\sqrt{2}} \\
 1 & \frac{1}{2} \left(1+\sqrt{5}\right) & 1 & \frac{1}{2} \left(1+\sqrt{5}\right) & -\frac{1+\sqrt{5}}{\sqrt{2}} & -\sqrt{2}
   \\
 \frac{1}{2} \left(1+\sqrt{5}\right) & -1 & \frac{1}{2} \left(1+\sqrt{5}\right) & -1 & \sqrt{2} &
   -\frac{1+\sqrt{5}}{\sqrt{2}} \\
 \frac{1+\sqrt{5}}{\sqrt{2}} & -\sqrt{2} & -\frac{1+\sqrt{5}}{\sqrt{2}} & \sqrt{2} & 0 & 0 \\
 \sqrt{2} & \frac{1+\sqrt{5}}{\sqrt{2}} & -\sqrt{2} & -\frac{1+\sqrt{5}}{\sqrt{2}} & 0 & 0 \\
\end{array}
\right)
\end{equation*}
\begin{equation*}
  T =   \left(
\begin{array}{cccccc}
 1 & 0 & 0 & 0 & 0 & 0 \\
 0 & e^{\frac{4 i \pi }{5}} & 0 & 0 & 0 & 0 \\
 0 & 0 & -1 & 0 & 0 & 0 \\
 0 & 0 & 0 & -e^{\frac{4 i \pi }{5}} & 0 & 0 \\
 0 & 0 & 0 & 0 & e^{\frac{27 i \pi }{40}} & 0 \\
 0 & 0 & 0 & 0 & 0 & -i e^{\frac{3 i \pi }{8}} \\
\end{array}
\right)
\end{equation*}
Since $\SU(2)_2 \gtensor \SU(2)_3$ contains the even part of $\SU(2)_3$ as a subcategory which is itself an MTC (Fibonacci), $\SU(2)_2 \gtensor \SU(2)_3$ must split. In fact, $\SU(2)_2 \gtensor \SU(2)_3 \simeq \text{Fib} \boxtimes \TLJ(-i e^{\frac{\pi i}{8}})$.
\end{example}

\section{Modular tensor categories from SOL geometry}\label{SOL}

\subsection{Character varieties of torus bundles over the circle}\label{subsec:character_torus}

One of the non-hyperbolic geometries is SOL and some examples of closed manifolds are torus bundles over the circle with Anosov monodromy maps.

Let $M$ be a torus bundle over $S^1$ with the monodromy map $\begin{pmatrix}a&b\\c&d\end{pmatrix}\in \SL(2,\mathbb{Z})$ where $|a+d|>2$. Its fundamental group has the presentation,
\begin{equation}\label{equ:torus_pi1}
\pi_1(M)=\langle x,y,h\ |\ x^ay^c=h^{-1}xh,\ x^by^d=h^{-1}yh,\ xyx^{-1}y^{-1}=1 \rangle,
\end{equation}
where $x$ and $y$ are the meridian and longitude, respectively, on the torus, and $h$ corresponds to a loop around the $S^1$ component. We consider non-Abelian characters of $M$ to $\SL(2,\bbC)$. Let $\rho: \pi_1(M) \to \SL(2,\bbC)$ be a non-Abelian representation.

First, we consider the case where $\rho(x)$ is diagonalizable. Up to conjugation, assume $\rho(x)$ is diagonal. Since $y$ commutes with $x$, $\rho(y)$ is also diagonal, and moreover, $\rho(x)$ and $\rho(y)$ cannot be both contained in the center $\{\pm I\}$. (Otherwise, the image of $\rho$ would be Abelian.) If $\rho(x) \neq \pm I$, it follows from the relation  $x^ay^c=h^{-1}xh$ that $\rho(h)$, up to conjugation, simply permutes the two eigenvectors of $\rho(x)$. The same conclusion is obtained if $\rho(y) \neq \pm I$. Hence, we may assume $\rho$ takes the following form (abbreviating $\rho(x)$ simply as $x$),
\begin{equation}\label{equ:torus_irrep_form}
    x=
    \begin{pmatrix}
    \alpha &0\\
    0&\alpha^{-1}
    \end{pmatrix},\quad 
    y=
    \begin{pmatrix}
   \beta&0\\
    0&\beta^{-1}
    \end{pmatrix},\quad
    h=
    \begin{pmatrix}
    0&1\\
    -1&0
    \end{pmatrix},
\end{equation}
where $\text{Im}(\alpha) \geq 0$ and either $\alpha \neq \pm 1$ or $\beta \neq \pm 1$. The presentation of $\pi_1(M)$ yields the following equations for $\rho$,
\begin{equation}\label{equ:torus_rep_equation1}
   \alpha^{a+1} \beta^c\ =\ \alpha^b \beta^{d+1} \ =\  1,
\end{equation}
from which we deduce the relations,
\begin{equation}
    \alpha^{a+d+2} \ = \ \beta^{a+d+2} \ = \ 1.
\end{equation}
Let $N = |a+d+2|$. Hence $\alpha$ and $\beta$ are both $N$-th root of unity. Set $\alpha = e^{\frac{2 \pi i\, k}{N}},\ \beta = e^{\frac{2 \pi i\, l}{N}}$ such that $0 \leq k \leq \frac{N}{2},\ 0 \leq l < N$, and either $k \neq 0, \frac{N}{2}$ or $l \neq 0, \frac{N}{2}$. Then, Equation \ref{equ:torus_rep_equation1} can be equivalently written as,
\begin{equation}\label{equ:torus_rep_equation2}
\begin{split}
    (a+1)\,k + c\, l &= 0 \mod N\\
    b\, k + (d+1)\,l &= 0 \mod N
\end{split}
\end{equation}
The solutions to Equation \ref{equ:torus_rep_equation2} depend on a number of conditions involving $a,\ b, \ c,$ and $d$. When at least one of $a+1,\ c,\ b,\ d+1$ is co-prime to $N$, there is a compact form to organize all the solutions. For instance, when $(c, N)$ are co-prime, the solutions are simply given by,
\begin{equation}
    l = -\tilde{c}(a+1)k \mod N, \quad k = 1, \cdots, \floor{\frac{N-1}{2}},
\end{equation}
where $\tilde{c}$ is the multiplicative inverse of $c$ in $\bbZ_N$. The representations thus obtained are all irreducible.

Now we consider the case where $\rho(x)$ is not diagonalizable. Then neither is $\rho(y)$ diagonalizable. Up to conjugation, we may assume that $\rho(x)$ and $\rho(y)$ are both upper triangular, each have a single eigenvalue $+1$ or $-1$ lying on the diagonal, and $\rho(h)$ is diagonal. Thus, $\rho$ takes the form,
\begin{equation}\label{equ:torus_reducible_form}
    x=(-1)^{\epsilon_x}
    \begin{pmatrix}
    1&1\\
    0&1
    \end{pmatrix},\quad
    y=(-1)^{\epsilon_y}
    \begin{pmatrix}
    1&u\\
    0&1
    \end{pmatrix},\quad
    h=
    \begin{pmatrix}
    v&0\\
    0&v^{-1}
    \end{pmatrix},
\end{equation}
where $\epsilon_x, \epsilon_y \in \{0,1\}$ and $u \neq 0$. From the presentation of $\pi_1(M)$, we deduce the equations to be satisfied,
\begin{equation}\label{equ:torus_rep_equation3}
\begin{split}
    (a+1)\,\epsilon_x + c\, \epsilon_y &= 0 \mod 2\\
    b\, \epsilon_x + (d+1)\,\epsilon_y &= 0 \mod 2
\end{split}
\end{equation}
\begin{equation}\label{equ:torus_rep_equation4}
\begin{split}
    c \, u^2 + (a-d) u - b = 0, \quad v^2 = \frac{1}{cu+a}.
\end{split}
\end{equation}
Equation \ref{equ:torus_rep_equation4} is equivalent to,
\begin{equation}\label{equ:torus_rep_equation5}
\begin{split}
    (v+v^{-1})^2 = a+d+2, \quad u = \frac{v^{-2}-a}{c}.
\end{split}
\end{equation}
From Equation \ref{equ:torus_rep_equation5}, we see that for each fixed $\epsilon_x$ and $\epsilon_y$, there are four inequivalent representations, but only two characters. We choose a representative for each character by setting,
\begin{equation}\label{equ:torus_rep_equation6}
    u = \frac{d-a + \sqrt{(a+d)^2-4}}{2c}, \quad v^2 = \frac{1}{cu+a} = \frac{a+d - \sqrt{(a+d)^2-4}}{2}.
\end{equation}

The solution set to Equation \ref{equ:torus_rep_equation3} depends on the parity of the entries of the monodromy matrix. Let $P$ be the quadruple that records the parity of the entries $(a,d;b,c)$ and we use $`e\textrm'$ to denote for `even' and $`o\textrm'$ for `odd'. For instance, $P = (e,e;o,e)$ means $b$ is odd and the rest are even.
The solutions contain the following possible values for $\epsilon_x$ and $\epsilon_y$,
\begin{itemize}
    \item $\epsilon_x = 0, \ \epsilon_y = 0$;
    \item $\epsilon_x = 1, \ \epsilon_y = 1$, only if $P = (e,e;o,o)$ or $P= (o,o; e,e)$;
    \item $\epsilon_x = 0, \ \epsilon_y = 1$, only if $P = (o,o;o,e)$ or $P= (o,o; e,e)$;
    \item $\epsilon_x = 1, \ \epsilon_y = 0$, only if $P = (o,o;e,o)$ or $P= (o,o; e,e)$.
\end{itemize}
Note that the last three cases above all imply that $N = |a+d+2|$ is even and all possible configurations of $P$ that have $N$ even are contained in one (or more) of the last three cases.

To summarize, the non-Abelian characters of $M$ contain two types, the irreducible and the reducible ones. The irreducible characters take the form of Equation \ref{equ:torus_irrep_form} and are determined by Equation \ref{equ:torus_rep_equation2}. The reducible characters take the form of Equation \ref{equ:torus_reducible_form} and are determined by Equation \ref{equ:torus_rep_equation6} and the possible values of $\epsilon_x$ and $\epsilon_y$ discussed above.

\subsection{Torsion and Chern-Simons invariant of torus bundles}\label{subsec:torus_torsion_CS}
In this subsection, we compute the torsion and Chern-Simons invariant for the torus bundle over the circle $M$ with the monodromy map $\begin{pmatrix}a&b\\c&d\end{pmatrix}\in \SL(2,\mathbb{Z})$ where $|a+d|>2$. Its fundamental group has a presentation given in Equation \ref{equ:torus_pi1}.

Construct a cell structure for $M$ as follows. See Figure \ref{fig:torus_cell}. The cell structure contains,
\begin{itemize}
    \item a single $0$-cell $v$;
    \item three $1$-cells corresponding to the generators $x, \ y$, and $h$ in the presentation of  $\pi_1(M)$;
    \item three $2$-cells corresponding to the three relations in the presentation of $\pi_1(M)$. Explicitly, denote them by $s_1, \ s_2$ and $s_3$ such that $\partial s_1 = yxy^{-1}x^{-1}$, $\partial s_2 = h^{-1}xh (x^ay^c)^{-1}$, and $\partial s_3 = h (x^by^d)h^{-1} y^{-1}$. Graphically, $s_1, \ s_2$ and $s_3$ correspond to the top face, the back face, and the left face, respectively, in Figure \ref{fig:torus_cell} with the induced orientation of the cube.
    \item a single $3$-cell $t$. Think of a $3$-cell as a cube. Then the attaching map is determined by the identification of faces described in Figure \ref{fig:torus_cell}.
\end{itemize}
\begin{figure}
    \centering
\begin{tikzpicture}[scale = 2, thick, >=stealth]
\draw[red,myarrow = {0.5}{>}] (0,0)--(1,0);
\draw[red,myarrow = {0.5}{>}] (0.5,0.5)--(1.5,0.5);
\draw[myarrow = {0.5}{>}] (0,-1)--(1,-1)node[pos = .5, anchor =  north]{$x^ay^c$};
\draw[myarrow = {0.5}{>},dashed] (0.5,-0.5)--(1.5,-0.5);

\draw[green, myarrow = {0.75}{>},myarrow = {0.65}{>},myarrow = {0.7}{>}] (0,0)--(0,-1);
\draw[green, myarrow = {0.75}{>},myarrow = {0.65}{>},myarrow = {0.7}{>}] (1,0)--(1,-1);
\draw[green, myarrow = {0.75}{>},myarrow = {0.65}{>},myarrow = {0.7}{>},dashed] (0.5,0.5)--(0.5,-0.5);
\draw[green, myarrow = {0.75}{>},myarrow = {0.65}{>},myarrow = {0.7}{>}] (1.5,0.5)--(1.5,-0.5);

\draw[blue,myarrow = {0.6}{>},myarrow = {0.68}{>}] (0.5,0.5)--(0,0);
\draw[blue,myarrow = {0.6}{>},myarrow = {0.68}{>}] (1.5,0.5)--(1,0);
\draw[myarrow = {0.6}{>},myarrow = {0.68}{>},dashed] (0.5,-0.5)--(0,-1);
\draw[myarrow = {0.6}{>},myarrow = {0.68}{>}] (1.5,-0.5)--(1,-1)node[pos = .5, anchor =  west]{$x^by^d$};

\draw[->] (2,-0.5)--(2,0.5);
\node at (2.6,0) {$\Phi = \begin{pmatrix}
a & b\\
c & d\\
\end{pmatrix}$};

\node at (0.5,0.6) {$v$};
\node at (1,0.6) {$x$};
\node at (-0,0.25) {$y$};
\node at (0.4,-0.3) {$h$};
\end{tikzpicture}
\caption{A cell structure for the torus bundle with monodromy matrix $\Phi$
For convenience but no other purposes, mark the vertical edges green, the horizontal on the top face red, and the $45^o$-slope edges on the top face blue. Edges of the same color and the same arrow are identified. The front and back faces are identified by the obvious map, and so are the left and right side faces. The bottom face is identified to the top via the monodromy map $\Phi$. Hence, the single-arrow edge and the double-arrow edge  at the bottom face are homotopic to $x^ay^c$ and $x^by^d$, respectively.} \label{fig:torus_cell}
\end{figure}
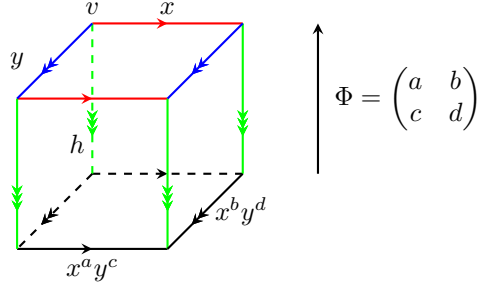

Let $V$ be a representation $\rho: \pi_1(M) \to GL(V)$, and let $\{v_j\ | \ j = 1, 2, \cdots\}$ be an arbitrary basis of $V$. We now construct the chain complex.
For simplicity, assume that $a,b,c,d\geq0$, $a\geq c$, $b\geq d$. Other cases can be dealt similarly. Fix an arbitrary preimage $\tilde{v}$ of $v$. For each other cell $\sigma$, fix a lifting $\tilde{\sigma}$ starting at the base point $\tilde{v}$. We have the following chain complex,
$$0\longrightarrow C_3\stackrel{\partial_3}{\longrightarrow}C_2\stackrel{\partial_2}{\longrightarrow}C_1\stackrel{\partial_1}{\longrightarrow}C_0\longrightarrow0
$$
where $C_i=C_i(\widetilde{M})\otimes_{\mathbb{Z}[\pi_1(M)]}V$. As a vector space, $C_i$ has the following basis, $C_3 = \text{span}\{\tilde{t} \otimes v_j\ | \ j = 1,2, \cdots\}$, $C_2 = \text{span}\{\tilde{s}_i \otimes v_j\ | \ i = 1,2,3,\  j = 1,2, \cdots\}$, $C_1 = \text{span}\{\tilde{\sigma} \otimes v_j\ | \ \sigma = x,y,h,\  j = 1,2, \cdots\}$, $C_0 = \text{span}\{\tilde{v} \otimes v_j\ | \ j = 1,2, \cdots\}$. We present the boundary map $\partial_i$ as a block matrix with each entry a $\dim(V) \times \dim(V)$ block. Also, denote $S: \bbZ[\pi_1(M)] \to \bbZ[\pi_1(M)]$ the antipode map that sends a group element $g \in \pi_1(M)$ to its inverse $g^{-1}$ and linearly extends to the whole ring. Lastly, for a matrix $A$ with entries in $\bbZ[\pi_1(M)]$, $\rho \circ S (A)$ is meant applying $\rho \circ S$ to every entry of $A$. With the above conventions, the boundary map is given by,

$$\partial_3=\rho \circ S \begin{pmatrix}
1-hw(x,y)\\
1-y\\
1-x
\end{pmatrix}
$$

$$\partial_2=\rho \circ S
\begin{pmatrix}
y-1 & 1-h\sum_{i=1}^{a-1}x^i & h\sum_{i=1}^{b-1}x^i\\
1-x & -hx^a\sum_{i=1}^{c-1}y^i & hx^b\sum_{i=1}^{d-1}y^i-1\\
0   & x-1&1-y
\end{pmatrix}
$$

$$\partial_1=\rho \circ S
\begin{pmatrix}
x-1&y-1&h-1
\end{pmatrix}
$$
where $w$ is a polynomial of $x,y$ with the sum of its coefficients equal to 1.

For each of the non-Abelian characters of $\pi_1(M)$ to $\SL(2,\bbC)$, we will compute its torsion below and show (implicitly) that the associated  chain complex is always acyclic and the torsion does not depend on the representation chosen in the equivalence class of a character.

For an irreducible representation $\rho$ given in Equation \ref{equ:torus_irrep_form} that satisfies Equation \ref{equ:torus_rep_equation2}, its adjoint representation has the form,
$$x=
\begin{pmatrix}
\alpha^2&0&0\\
0&1&0\\
0&0&\alpha^{-2}
\end{pmatrix},\  
y=
\begin{pmatrix}
\beta^2&0&0\\
0&1&0\\
0&0&\beta^{-2}
\end{pmatrix},\ h=
\begin{pmatrix}
0&0&-1\\
0&-1&0\\
-1&0&0
\end{pmatrix}
$$
Denote by $I$ and $O$ and $3 \times 3$ identity matrix and zero matrix, respectively, and let
$$
A=\begin{pmatrix}1&0&0\\0&0&0\\0&1&0\end{pmatrix}, B=\begin{pmatrix}0&0&0\\0&0&1\\0&0&0\end{pmatrix}.
$$
Define the block matrices,
$$
K_1 = 
\begin{pmatrix}
A \\
O \\
B\\
\end{pmatrix}, \  
K_2 =
\begin{pmatrix}
O & A \\
I & O\\
O & B\\
\end{pmatrix}, \ 
K_3 = 
\begin{pmatrix}
I 
\end{pmatrix}.
$$
It can be checked directly that the columns (as vectors in $C_{i-1}$) of $\partial_iK_i$ is a basis of $\Ima(\partial_i)$. Set $K_4 = K_0$ to be the empty matrix. Now for $i = 0,1,2,3$, let
\begin{equation*}
A_i = 
\setlength{\dashlinegap}{2pt}
\left(
\begin{array}{c:c}
\partial_{i+1}K_{i+1}  & K_i
\end{array} \right),
\end{equation*}
then the columns of $A_i$ give a basis for $C_i$.  By direct calculations, we obtain the torsion,
\begin{equation*}
    \Tor(\rho) \ =\  \left|\frac{\det(A_1)\det(A_3)}{\det(A_0)\det(A_2)}\right|\ =\ \frac{|a+d+2|}{4}.
\end{equation*}

Now we compute the torsion of the reducible representations $\rho$ given in Equation \ref{equ:torus_reducible_form}. The associated adjoint representation takes the form,
\begin{equation*}
    x = 
\left(
\begin{array}{ccc}
 1 & -2 & -1 \\
 0 & 1 & 1 \\
 0 & 0 & 1 \\
\end{array}
\right),\ 
y = 
\left(
\begin{array}{ccc}
 1 & -2 u & -u^2 \\
 0 & 1 & u \\
 0 & 0 & 1 \\
\end{array}
\right),\ 
h = 
\left(
\begin{array}{ccc}
 v^2 & 0 & 0 \\
 0 & 1 & 0 \\
 0 & 0 & \frac{1}{v^2} \\
\end{array}
\right),
\end{equation*}
which are clearly independent on the sign terms $\epsilon_x$ and $\epsilon_y$. Let,
\begin{equation*}
\begin{split}
    A=\begin{pmatrix}0&0&0\\1&0&0\\0&1&0\end{pmatrix},\  B=\begin{pmatrix}0&0&0\\0&0&1\\0&0&0\end{pmatrix},\  C=\begin{pmatrix}0&0&0\\0&0&0\\1&0&0\end{pmatrix}, \\
    D=\begin{pmatrix}0&0&0\\0&1&0\\0&0&1\end{pmatrix},\  E=\begin{pmatrix}0&0&0\\0&0&0\\0&0&1\end{pmatrix},\  F=\begin{pmatrix}1&0&0\\0&0&0\\0&1&0\end{pmatrix}.
\end{split}
\end{equation*}
Define the block matrices,
$$
K_1 = 
\begin{pmatrix}
E \\
O \\
F\\
\end{pmatrix}, \  
K_2 =
\begin{pmatrix}
A & O \\
B & C\\
O & D\\
\end{pmatrix}, \ 
K_3 = 
\begin{pmatrix}
I 
\end{pmatrix}.
$$
The matrices $K_i$ have the same properties as outlined in the case of irreducible representations above, and in the same way define the matrices $A_i$. It can be computed that,
\begin{equation*}
    \Tor(\rho) \ =\  \left|\frac{\det(A_1)\det(A_3)}{\det(A_0)\det(A_2)}\right|\ =\ |a+d+2|.
\end{equation*}
Some details for the derivation are as follows, where the condition $cu^2+(a-d)u-b=0$ is used to simplify expressions,
\begin{align*}
\Tor(\rho)&=|\frac{(2cu+a-d)(b-u+du)(a-b+1+(c-d-1)u)}{u(1-cu-a)^2(u-1)}|
\\
&=|\frac{(2cu+a-d)(b-u+du)(a-b+1+(c-d-1)u)}{(cu^2+(a-1)u)(cu^2+(a-1-c)u-a+1)}|
\\
&=|\frac{(2cu+a-d)(b-u+du)(a-b+1+(c-d-1)u)}{((d-1)u+b)((d-c-1)u+b-a+1)}|
\\
&=|\frac{(d-c-1)u+b-a+1}{2(c-d-1)cu^2+(2c(a-b+1)+(a-d)(c-d-1))u+(a-d)(a-b+1)}|
\\
&=|\frac{(2c(a-b+1)-(a-d)(c-d-1))u+(a-d)(a-b+1)+2b(c-d-1)}{(d-c-1)u+b-a+1}|
\\
&=|\frac{(a+d+2)((d-c-1)u+b-a+1)}{(d-c-1)u+b-a+1}|
\\
&=|a+d+2|.
\end{align*}

Now, we compute the CS invariant of $M$. Any irreducible representation of $\pi_1(M)$ to $\SL(2,\bbC)$ can be conjugated to one into $\SU(2)$ (see Equation \ref{equ:torus_irrep_form}), and Kirk and  Klassen computed its CS invariant in \cite{kirk90}. Here we use methods in Section \ref{subsec:CS_review} to compute the CS invariant of both irreducible and reducible but indecomposable ones, the latter of which can not be conjugated to $\SU(2)$.

Let $T_i \ (i =A,B)$ be two copies of the torus, and $I$ be the interval $[0,1]$. Then $M$ is obtained by gluing the two $T_i \times I$ such that $T_B \times \{0\}$ is glued to $T_A \times \{1\}$ via the identity map and $T_B \times \{1\}$ is glued to $T_A \times \{0\}$ via the map $\begin{pmatrix}a&b\\c&d\end{pmatrix}$. Let $(\mu_i, \lambda_i)$ be a positive basis of $H_1(T_i)$ so that, under the embedding $T_i \times I \hookrightarrow M$, $\mu_i$ and $\lambda_i$ are sent to $x$ and $y$, respectively. For $\kappa = 0,1$, denote by $\mu_i^{\kappa}$ the element of $H_1(T_i \times \{\kappa\})$ that corresponds to $\mu_i$ in $H_1(T_i \times I)$, and by $\lambda_i^{\kappa}$ in a similar way. Then $(\mu_i^{1}, \lambda_i^{1})$ is a positive basis for $H_1(T_i \times \{1\})$ and $(-\mu_i^{0}, \lambda_i^{0})$ is a positive basis for $H_1(T_i \times \{0\})$. These basis are identified as follows,
\begin{equation*}
    (\mu_B^0, \lambda_B^0) = (\mu_A^1, \lambda_A^1), \qquad
    (\mu_B^1, \lambda_B^1) = 
    (\mu_A^0, \lambda_A^0)\,
    \begin{pmatrix}a&b\\c&d\end{pmatrix}.
\end{equation*}

Set $N = |a+d+2|$. For an irreducible representation $\rho$ in Equation \ref{equ:torus_irrep_form} where $\alpha = e^{\frac{2\pi i\,k}{N}}$ and $\beta = e^{\frac{2\pi i\,l}{N}}$, we have
\begin{align*}
c_{T_i\times I}(\rho)&=[\frac{k}{N},\frac{l}{N},\frac{k}{N},\frac{l}{N};1]_{(\mu_i^1,\lambda_i^1),(\mu_i^0, \lambda_i^0)} \\
&=[\frac{k}{N},\frac{l}{N},-\frac{k}{N},\frac{l}{N};1]_{(\mu_i^1,\lambda_i^1),(-\mu_i^0, \lambda_i^0)} 
\end{align*}
Hence,
\begin{align*}
    &\ c_{T_A\times I}(\rho)\\
    &= [\frac{k}{N},\frac{l}{N},\frac{k}{N},\frac{l}{N};1]_{(\mu_A^1,\lambda_A^1),(\mu_A^0, \lambda_A^0)} \\
    &=[\frac{k}{N},\frac{l}{N},\frac{ak +cl}{N},\frac{bk + dl}{N};1]_{(\mu_A^1,\lambda_A^1),(\mu_B^1, \lambda_B^1)} \\
    &=[\frac{k}{N},\frac{l}{N},-\frac{k}{N},\frac{bk + dl}{N};\exp(2\pi i (-\nu)\frac{bk+dl}{N})] ,\ ( \nu := \frac{(a+1)k+cl}{N})\\
    &=[\frac{k}{N},\frac{l}{N},-\frac{k}{N},-\frac{l}{N};\exp(2\pi i (-\nu)\frac{bk+dl}{N} + 2\pi i (-\mu)\frac{k}{N})] ,\ (\mu := \frac{bk+(d+1)l}{N}) \\
    &=[\frac{k}{N},\frac{l}{N},-\frac{k}{N},\frac{l}{N};\exp(2\pi i f)]_{(\mu_A^1,\lambda_A^1),(-\mu_B^1, \lambda_B^1)}
\end{align*}
where,
\begin{align*}
    f = \nu\frac{bk+dl}{N} +  \mu\frac{k}{N} = \frac{k \mu - l \nu}{N} + \mu \nu.
\end{align*}
Note that, by Equation \ref{equ:torus_rep_equation2}, $\mu$ and $\nu$ are both integers. Also,
\begin{align*}
    c_{T_B\times I}(\rho)
    &= [\frac{k}{N},\frac{l}{N},-\frac{k}{N},\frac{l}{N};1]_{(\mu_B^1,\lambda_B^1),(-\mu_B^0, \lambda_B^0)} \\
    &= [\frac{k}{N},\frac{l}{N},-\frac{k}{N},\frac{l}{N};1]_{(\mu_B^1,\lambda_B^1),(-\mu_A^1, \lambda_A^1)} 
\end{align*}
By taking the pairing on $c_{T_A\times I}(\rho)$ and $c_{T_B\times I}(\rho)$, we obtain that,
\begin{equation}\label{equ:torus_CS_irrep}
    \CS(\rho)\ =\ f \ = \ \frac{k \mu - l \nu}{N}.
\end{equation}

For reducible representations $\rho_{\epsilon_x, \epsilon_y}$ in Equation \ref{equ:torus_reducible_form} depending on the values of $\epsilon_x$ and $\epsilon_y$ (see Section \ref{subsec:character_torus}), the computation of the CS invariant proceeds in the exactly the same way as for irreducible representations by making the substitution,
\begin{align*}
    \frac{k}{N} \to \frac{\epsilon_x}{2}, \quad \frac{l}{N} \to \frac{\epsilon_y}{2}.
\end{align*}
Consequently, by setting
\begin{align*}
    \nu = \frac{(a+1)\epsilon_x+c\epsilon_y}{2}, \quad \mu = \frac{b\epsilon_x+(d+1)\epsilon_y}{2},
\end{align*}
we obtain that,
\begin{align}
    \CS(\rho_{\epsilon_x, \epsilon_y}) &= \frac{\epsilon_x \mu - \epsilon_y \nu}{2} = \frac{\epsilon_x \mu + \epsilon_y \nu}{2} \nonumber \\
    &= \frac{(a+d+2)\epsilon_x\epsilon_y + b\epsilon_x + c\epsilon_y}{4}\label{equ:torus_CS_reducible00}
\end{align}
It can be checked that $\CS(\rho_{\epsilon_x, \epsilon_y}) \in \frac{1}{2}\bbZ$.

\subsection{Modular data from torus bundles over the circle} \label{subsec:MTC_torus}

In this subsection, let $M$ be a torus bundle over the circle with the monodromy map given by a matrix,
\begin{equation*}
    \begin{pmatrix}
    a & b\\
    c & d
    \end{pmatrix} \ \in \ \SL(2, \bbZ).
\end{equation*}
We assume that $N:= a + d+2 > 4$ is odd and $(c,N)$ are co-prime. It is direct to see that $b$ and $c$ are both odd, while $a$ and $d$ have different parity. Set $N = 2r+1$. Denote by $\tilde{c} \in \bbZ_N$ the multiplicative inverse of $c$ in $\bbZ_N$.

The non-Abelian character variety of $M$ to $\SL(2,\bbC)$ consists of the representations $\nabchar(M) = \{\rho_{+}, \rho_{-}, \rho_k, k = 1, \cdots, r\}$ which are defined as follows.
For $\rho_{\pm}$,
\begin{equation}
    x\mapsto
    \begin{pmatrix}
    1&1\\
    0&1
    \end{pmatrix},\quad
    y\mapsto
    \begin{pmatrix}
    1&u\\
    0&1
    \end{pmatrix},\quad
    h\mapsto
    \begin{pmatrix}
    v_{\pm}&0\\
    0&v_{\pm}^{-1}
    \end{pmatrix}
\end{equation}
where 
\begin{equation}
    u = \frac{d-a + \sqrt{(a+d)^2-4}}{2c}, \quad v_{\pm} = \pm \frac{1}{\sqrt{cu+a}}. 
\end{equation}
For $\rho_k$, $k = 1, \cdots, r$,
\begin{equation}
    x\mapsto
    \begin{pmatrix}
    e^{\frac{2\pi i k}{N}}&0\\
    0&e^{-\frac{2\pi i k}{N}}
    \end{pmatrix},\ 
    y\mapsto
    \begin{pmatrix}
    e^{\frac{-2\pi i \tilde{c}(a+1)k }{N}}&0\\
    0&e^{\frac{2\pi i \tilde{c}(a+1)k }{N}}
    \end{pmatrix},\ 
    h\mapsto
    \begin{pmatrix}
    0&1\\
    -1&0
    \end{pmatrix}
\end{equation}

In Section \ref{subsec:character_torus}, we computed the adjoint torsion and CS of representations of $\pi_1(M)$. In particular, it implies that all non-Abelian characters are adjoint-acyclic and their CS invariants are all rational numbers. As with the example of SFSs, we choose the candidate label set $L(M) = \nabchar(M)$. According to Section \ref{subsec:character_torus}, the torsion of these representations are given by
\begin{equation}\label{equ:torus_torsion}
    \Tor(\rho_{\pm}) = N, \quad \Tor(\rho_k) = \frac{N}{4}.
\end{equation}
The Chern-Simons invariant of $\rho_{\pm}$ is $0$ by Equation \ref{equ:torus_CS_reducible00}.
\begin{lemma}\label{lem:torus_CS}
 For $k = 1, \cdots, r$, the Chern-Simons invariant of $\rho_k$ is given by,
 \begin{equation}
     \CS(\rho_k) = -\frac{\tilde{c}k^2}{N}.
 \end{equation}
 \begin{proof}
 This can be derived from Equation \ref{equ:torus_CS_irrep}.
 \end{proof}
\end{lemma}

We will show below that the premodular categories obtained from the torus bundles are related to quantum group categories associated with $\so_{2r+1}$.

For an odd integer $N = 2r+1 > 0$, let $\so_N$ (Type $B$) be the Lie algebra of $\textrm{SO}(N)$. Given $q = e^{\frac{m \pi i }{2N}}$ such that $q^2$ is a primitive $2N$-th root of unity (thus $m$ is odd and $(m,N)$ are co-prime), there is an associated premodular category $\C(\so_N, q, 2N)$ of rank $r+4$. See \cite{naidu2011finiteness} and references therein. When $m=1$, the corresponding category is always an MTC, and is denoted by  $\textrm{SO}(N)_2$ in physics literature. The MTC has the label set,
\begin{equation}\label{equ:so_label_set}
    \{\unitobj, Z\} \sqcup \{Y_1, \cdots, Y_r\} \sqcup \{X_1, X_2\}.
\end{equation}
We will mainly be interested in the (adjoint) monoidal subcategory $\C(\so_N, q, 2N)_{ad}$ linearly spanned by the objects $\unitobj, Z, Y_1, \cdots, Y_r$. So only modular data on this subcategory is given below.

The twists are,
\begin{equation}\label{equ:so_twist}
    \theta_{\unitobj} = \theta_{Z} = 1, \quad \theta_{Y_k} = q^{2(Nk-k^2)},\  k = 1, \cdots, r.
\end{equation}
The un-normalized $S$-matrix is,
\begin{equation}\label{equ:so_Smatrix1}
    \tilde{S}_{\alpha\beta} = 
    \begin{cases}
     1   & \alpha \in \{\unitobj, Z\}, \beta \in \{\unitobj, Z\} \\
     2   &  \alpha \in \{\unitobj, Z\}, \beta \in \{Y_1, \cdots,Y_r\}
    \end{cases}
\end{equation}
\begin{equation}\label{equ:so_Smatrix2}
    \tilde{S}_{kj}:= \tilde{S}_{Y_kY_j} = 2(q^{4kj} + q^{-4kj}) = 4 \cos \frac{2\pi m\,kj}{N}.
\end{equation}
In particular, there are only two values for quantum dimensions, $d_{\unitobj} = d_Z = 1$ and $d_{k}:= d_{Y_k} = 2$. The total dimension is $D = \sqrt{2N}$. Note that $\C(\so_N, q, 2N)_{ad}$ is a proper premodular category of rank $r+2$.
\begin{remark}
The label set as ordered in Equation \ref{equ:so_label_set} correspond to the labels $\{\textbf{0}, 2\lambda_1, \lambda_1, \cdots, \lambda_{r-1}, 2\lambda_r,\lambda_r, \lambda_r+\lambda_1 \}$ in \cite{naidu2011finiteness}. Although the $S$-matrix in \cite{naidu2011finiteness} is only given for the root $q = e^{\frac{\pi i}{2N}}$, the case for other roots can be easily deduced by either applying a Galois action to the original $S$-matrix or using the formula
\begin{equation*}
    \tilde{S}_{\lambda\mu} = \theta_{\lambda}^{-1}\theta_{\mu}^{-1} \sum_{\nu} N_{\lambda^*\mu}^{\nu}\theta_{\nu}d_{\nu}.
\end{equation*}
\end{remark}

Now, for the torus bundle defined at the beginning of the subsection, recall that $N = a+d+2$ is odd, and $\tilde{c}c = 1 \in \bbZ_N$. 
Let $m = -2 \tilde{c} - N \in \bbZ$ which is well defined up to multiples of $2N$. For clarity,  fix an arbitrary representative for $m$, and let $q = e^{\frac{m \pi i}{2N}}$. Note that $m$ is odd and co-prime to $2N$. Hence $q^2$ is a primitive $2N$-th root of unity.

We propose the following correspondence between $\nabchar(M)$ and loop operators,
\begin{equation}\label{equ:torus_local}
\begin{split}
    \rho_{\pm} &\mapsto (x, \Sym^0),\\
    \rho_{k}   &\mapsto (x^{mk},\Sym^1).
\end{split}
\end{equation}
and designate $\rho_{+}$ as the unit object, 
\begin{align}\label{equ:torus_unit}
    \rho_{+} = \unitobj &\mapsto (x, \Sym^0).
\end{align}

\begin{theorem}\label{thm:torus_MTC}
Let $M$ be the torus bundle over the circle with the monodromy matrix 
$\begin{pmatrix}
a & b\\
c & d
\end{pmatrix}$
such that $N = a+d+2 > 4$ is odd and $(c,N)$ are co-prime. With the choice of loop operators and unit object in Equations \ref{equ:torus_local} and \ref{equ:torus_unit}, respectively, and $q$ as above, the modular data constructed from $M$ matches that of $\mathcal{C}(\so_N, q, 2N)_{ad}$, the adjoint subcategory of $\mathcal{C}(\so_N, q, 2N)$. 
\begin{proof}
For convenience, we also write $\rho_{\pm}$ and $\rho_k$ simply as $\pm$ and $k$, respectively. The correspondence between $\nabchar(M)$ and label set of $\mathcal{C}(\so_N, q, 2N)_{ad}$ is,
\begin{equation*}
    \rho_{+} \leftrightarrow \unitobj, \quad \rho_{-} \leftrightarrow Z, \quad \rho_{k} \leftrightarrow Y_k, \ k = 1, \cdots, r.
\end{equation*}
We first check the twists. By Equation \ref{equ:so_twist}, 
\begin{align*}
    \theta_{Y_k} \ =\  q^{2(Nk-k^2)} \ =\  e^{-\frac{2\pi i}{N} \frac{Nk-k^2}{2}(2\tilde{c}+N)} \ = \ e^{2\pi i \frac{\tilde{c}k^2}{N}} .
\end{align*}
Note that in the last equality, we used the fact that $(Nk-k^2)/2$ is an integer. By Lemma \ref{lem:torus_CS}, we immediately have
\begin{align*}
    \theta_{Y_k} = e^{-2\pi i \CS(\rho_k)}.
\end{align*}
Of course, for $\rho_{\pm}$, a similar relation to the above  holds trivially.

Next, we verify quantum dimension.
\begin{align}
    W_{+}(\pm) = 1, \quad W_{+}(k) = \Tr_{\Sym^1}(\rho_{+}(x^{mk})) =2.
\end{align}
This means that the total dimension is $D = \sqrt{2N}$ (equal to the dimension of $\mathcal{C}(\so_N, q, 2N)_{ad}$), and by Equation \ref{equ:torus_torsion}, for each $\rho \in \nabchar(M)$, the normalized quantum dimension matches the torsion,
\begin{align*}
    \frac{W_{+}(\rho)}{D} = (2\Tor(\rho))^{-\frac{1}{2}}.
\end{align*}
Lastly, for the $S$-matrix computed from the $W$ matrix,
\begin{align*}
    \tilde{S}_{\alpha\beta} = 1,\quad \alpha, \beta \in \{+, -\}. 
\end{align*}
\begin{align*}
    &\tilde{S}_{\alpha k} = W_k(\alpha)W_{+}(k) = 2,\\
    &\tilde{S}_{k\alpha} = W_{\alpha}(k)W_{+}(\alpha) = 2, \quad \alpha \in \{+, -\}.
\end{align*}
\begin{align*}
    \tilde{S}_{kj} = W_j(k)W_+(j) = 2\, \Tr_{\Sym^1}(\rho_j(x^{mk})) = 4 \cos \frac{2\pi m\, kj}{N}, \quad k,j = 1, \cdots, r.
\end{align*}
This matches the $S$-matrix of $\mathcal{C}(\so_N, q, 2N)_{ad}$ in Equations \ref{equ:so_Smatrix1} and \ref{equ:so_Smatrix2}.
\end{proof}
\end{theorem}

Torus bundles  are not $\mathbb{Z}_2$ homology sphere by Equation \ref{equ:torus_pi1}, and  the adjoint subcategory of $\mathcal{C}(\so_N, q, 2N)$ is a properly premodular category by Equation \ref{equ:so_Smatrix1}. Hence the above theorem verifies the conjecture on the non-degeneracy of the resulting premodular category for torus bundles considered in this paper.

\begin{remark}
In this subsection, we restricted ourselves to the case where $N = a+d+2>4$ is odd and $(c, N)$ are co-prime. In other cases, it seems less straightforward to derive the character variety and the structure of the character variety depends on the parity of $N$ (among other factors). This is expected, since we conjecture in the general case the corresponding premodular category is also related to the adjoint subcategory of some $\C(\so_N, q,l)$ whose structure varies dramatically depending on the parity of $N$ and the value of $N$ modulo $4$ in the case of even $N$. We leave this as a future direction.
\end{remark}

\section{Full data of modular categories and beyond}

The structure theory of MTCs is naturally divided into two parts: one is the classification of modular data (MD), and the other is for a fixed modular data, the classification of modular isotopes (MIs)\footnote{A terminology due to C. Delaney: distinct MTCs with the same MD are called modular isotopes of each other.}.  The missing steps in the program from three manifolds to MTCs are then an algorithm to define loop operators for an admissible candidate label set, hence a candidate MD, and the $F$-matrices for the fusion structures beyond MD.

Physics point to a framework that is a generalization of gauging finite group symmetries \cite{gaugingP,gaugingM} to continuous non-Abelian Lie group symmetries such as $\SU(2)$.  One hint from physics is the form of the primitive loop operators in this paper: a pair $(a,R)$, where $a$ is a conjugacy class of the fundamental group, some kind of flux, and $R$ is an irreducible representation of $\SU(2)$, some charge of the $\SU(2)$ symmetry.  The $F$-matrices are difficult to find, so we wonder if they depend on more than topology: some geometric information of the given three manifolds.

\subsection{Towards the full data}

\subsubsection{From non-Abelian characters to loop operators}

The identification of a simple object type with a non-Abelian character is based on the relation between a simple object type and a loop operator in the solid torus. In a $(2+1)$-TQFT, the rank of an MTC is the same as the dimension of the vector space $V(T^2)$ associated to the torus $T^2$ from the TQFT.  One basis $\{e_a\}$ of the vector space $V(T^2)$ consists of labeled core curves of a solid torus by a complete representative set of simple objects $\{a\}$.  Then each basis element $e_a$ can be obtained as the image of a loop operator $O_a$ on $e_0$---the basis element associated to the vacuum, i.e. $|e_a>=O_a|e_0>$.

Suppose a non-Abelian character corresponds to a primitive loop operator $(a,R)$ of the three manifold $X$.  Then $a$ can be represented by a knot $K_a$ in $X$.  The knot complement of $K_a$ in $X$ determines a vector in $V(T^2)$ from the reduction of 6d SCFT onto $X$, which should be related to $e_a$, hence a simple object type eventually.

\subsubsection{From flatness equations to pentagons}

One possible relation between pentagon equations and flatness is that the flatness of $\SL(2,\bbC)$-connections corresponding to the fundamental group representations can be translated into pentagon equations for the $F$-matrices.  It is known that pentagon equations can be interpreted as flatness equations for bi-unitary connections on finite graphs (see e.g. \cite{yasu20}).

\subsection{Towards gauging $\SU(2)$ R-symmetry}

An R-symmetry of a super-symmetric theory is an outer automorphism of the super-Poincare group that fixes the Poincare group.  It is pointed out in \cite{gang20} that the R-symmetries in infrared could be different from those in ultra-violet.  Hence we could have an $\SU(2)$ R-symmetry for the residual topological theory in infrared, which is probably often trivial.  We believe that the MTCs obtained from three manifolds in this program are actually the results of gauging such $\SU(2)$ R-symmetries of the residual topological theory in infrared, which generalizes gauging of finite group symmetries \cite{gaugingP,gaugingM}.

\subsection{Towards quantum double of infinite discrete groups}

An interesting class of MTCs comes from the representation categories of quantum doubles of finite groups.  A naive generalization to infinite discrete groups does not work.  The program in this paper can be regarded as a first step in this direction for the class of 3-manifold groups.  The choice of the simple Lie group serves as an analogue of a level in quantum groups.

\subsection{Climbing the dimension ladder}

Two interesting classes of quantum algebras are vertex operator 
algebras (VOAs) and MTCs.  
The bulk-edge correspondence of topological phases of
matter makes them into a unified theory of two and three dimensions. 
The program in this paper suggests an inversion of dimensions: 
MTCs and VOAs could also fit into a unified theory of three and four
dimensional manifolds, where 4-manifolds with 3-manifold boundaries could
give rise to VOAs that realize the boundary MTCs.

\subsection{Open questions}

There are many other interesting open questions in this program.  One obvious one is to extend our results to more examples such as Seifert fibered spaces with more than three fibers and the remaining cases of our torus bundles over the circle examples.  It is also not clear how to obtain MTCs which are not self-dual.  As mentioned in Sec. \ref{sec:backgrounds}, representations of $\SL(2,\bbC)$ come in group of four and a natural guess is that one of the four is the dual anyon type.  If so, then which one? The dual representation is a candidate.  Another general direction is what operations on MTCs that  standard topological constructions of three manifolds such as connected sum and torus decomposition correspond to.  Connect sum should correspond to Deligne product.

Our adjoint-acyclic condition for a representation $\rho$ is closely related to $H^1(\pi, \text{Adj}\circ \rho)=0$.  Are they equivalent? It should be equivalent for irreducible representations, but for the indecomposable reducible ones, it is not clear.

\vspace{0.5cm}
\noindent \textbf{Acknowledgments.} 
Y. Q. and Z.W. are partially supported by NSF grant
FRG-1664351 and CCF 2006463. S.-X. C. is partially supported by NSF CCF 2006667. The research is also partially supported by ARO MURI contract W911NF-20-1-0082.  The third author thanks Dongmin Gang for helpful communications, who pointed out that the Seifert fibered spaces $(3,3,r)$ would give rise to modular tensor categories related to $SU(2)_k$. The first author thanks Eric Rowell for clarifying some facts about $\mathcal{C}(\mathfrak{so},q,l)$.

\bibliographystyle{plain}
\bibliography{MTC_bib}

\begin{thebibliography}{10}

\bibitem{auckly94}
David~R. Auckly.
\newblock Topological methods to compute {C}hern-{S}imons invariants.
\newblock {\em Mathematical Proceedings of the Cambridge Philosophical
  Society}, 115(2):229 -- 251, 1994.

\bibitem{gaugingP}
Maissam Barkeshli, Parsa Bonderson, Meng Cheng, and Zhenghan Wang.
\newblock Symmetry fractionalization, defects, and gauging of topological
  phases.
\newblock {\em Physical Review B}, 100(11):115147, 2019.

\bibitem{bonderson2007non}
Parsa~Hassan Bonderson.
\newblock {\em Non-{A}belian anyons and interferometry}.
\newblock PhD thesis, California Institute of Technology, 2007.

\bibitem{gang20}
Gil~Young Cho, Dongmin Gang, and Hee-Cheol Kim.
\newblock M-theoretic genesis of topological phases.
\newblock {\em Journal of High Energy Physics}, 2020(11):1--58, 2020.

\bibitem{gaugingM}
Shawn~X Cui, C{\'e}sar Galindo, Julia~Yael Plavnik, and Zhenghan Wang.
\newblock On gauging symmetry of modular categories.
\newblock {\em Communications in Mathematical Physics}, 348(3):1043--1064,
  2016.

\bibitem{cui2021torus}
Shawn~X Cui, Paul Gustafson, Yang Qiu, and Qing Zhang.
\newblock From torus bundles to particle-hole equivariantization.
\newblock {\em accepted by Letters in Mathematical Physics}, arXiv:2106.01959,
  2021.

\bibitem{culler83}
Marc Culler and Peter~B. Shalen.
\newblock Varieties of group representations and splittings of 3-manifolds.
\newblock {\em Annals of Mathematics Second Series}, 117(1):109--146, 1983.

\bibitem{freed92}
Daniel~S. Freed.
\newblock Reidemeister torsion, spectral sequences, and {B}rieskorn spheres.
\newblock {\em J. renie angew Math}, 1992(429):75--89, 1992.

\bibitem{metaplectic}
Matthew~B Hastings, Chetan Nayak, and Zhenghan Wang.
\newblock On metaplectic modular categories and their applications.
\newblock {\em Communications in Mathematical Physics}, 330(1):45--68, 2014.

\bibitem{yasu20}
Yasuyuki Kawahigashi.
\newblock A remark on matrix product operator algebras, anyons and subfactors.
\newblock {\em Letters in Mathematical Physics}, pages 1--10, 2020.

\bibitem{kirk90}
Paul Kirk and Eric Klassen.
\newblock Chern-{S}imons invariants of 3-manifolds and representation spaces of
  knot groups.
\newblock {\em Math. Ann}, 287(343–367):109--146, 1990.

\bibitem{kirk93}
Paul Kirk and Eric Klassen.
\newblock Chern-{S}imons invariants of 3-manifolds decomposed along tori and
  the circle bundle over the representation space of ${T}^2$.
\newblock {\em Comm. Math. Phys}, 153(3):521--557, 1993.

\bibitem{kitano94}
Teruaki Kitano.
\newblock Reidemeister torsion of {S}eifert fibered spaces for
  ${SL}(2;\mathbf{C})$-representations.
\newblock {\em Tokyo J. Math}, 17(1):59--75, 1994.

\bibitem{milnor66}
J~Milnor.
\newblock Whitehead torsion.
\newblock {\em Bull. Amer. Math. Soc}, 72(3):358--426, 1966.

\bibitem{naidu2011finiteness}
Deepak Naidu and Eric~C Rowell.
\newblock A finiteness property for braided fusion categories.
\newblock {\em Algebras and representation theory}, 14(5):837--855, 2011.

\bibitem{orlikbook}
Peter Orlik.
\newblock {\em Seifert manifolds}, volume 291.
\newblock Springer, 2006.

\bibitem{scott83}
Peter Scott.
\newblock The geometry of 3-manifolds.
\newblock {\em Bull. London Math. Soc.}, 15:401--487, 1983.

\bibitem{turaev01}
Vladimir Turaev.
\newblock {\em Introduction to combinatorial torsions}.
\newblock Springer Science \& Business Media, 2001.

\bibitem{turaevbook}
Vladimir~G Turaev.
\newblock {\em Quantum invariants of knots and 3-manifolds}, volume~18.
\newblock Walter de Gruyter GmbH \& Co KG, 2020.

\bibitem{wangbook}
Zhenghan Wang.
\newblock {\em Topological quantum computation}.
\newblock Number 112. American Mathematical Soc., 2010.

\end{thebibliography}
\end{document}